\documentclass[11pt,english]{smfart}

\usepackage{amsmath,leftidx,amsthm,sfheaders}
\usepackage[all]{xy}
\usepackage{xspace,smfthm}
\usepackage[psamsfonts]{amssymb}
\usepackage{enumerate}
\usepackage{graphicx,color}
\usepackage{hyperref,fancyhdr}
\usepackage{tikz-cd}
\usepackage{xcolor}
\usepackage[utf8]{inputenc}
\usepackage[french]{babel}

\newtheorem*{remark}{\bf Remark}
\newtheorem*{question}{\bf Question}

\newtheorem{theorem}{\bf Theorem}[section]
\newtheorem{proposition}[theorem]{\bf Proposition}
\newtheorem{definition}[theorem]{\bf Definition}
\newtheorem{Theorem}{\bf Theorem}

\newtheorem{lemma}[theorem]{\bf Lemma}
\newtheorem{corollary}[theorem]{\bf Corollary}
\newtheorem*{conjecture}{\bf Conjecture}
\newtheorem{Corollary}{\bf Corollary}

\def\C{{\mathbb C}}
\def\N{{\mathbb N}}
\def\R{{\mathbb R}}
\def\Z{{\mathbb Z}}

\def\B{\mathbb{B}}
\def\D{\mathbb{D}}

\def\J{\mathcal{J}}

\def\P{\mathbb{P}}

\def\supp{\textup{supp}}

\def\pe{\textup{ := }}
\def\rat{\textup{Rat}}

\def\bif{\textup{bif}}

\def\J{\mathcal{J}}

\def\un{{\underline{n}}}

\def\and{{\quad\text{and}\quad}}

\newcommand\rlem[1]{{Lemma~\ref{lem#1}}}

\def\Xi{\varphi}

\newcommand\requ[1]{\eqref{equ#1}}

\def\J{\mathcal{J}}

\def\un{\underline{n}}

\def\eps{\varepsilon}

\def\ga{\gamma}
\def\ka{\kappa}
\def\Ga{\Gamma}
\def\Si{\Sigma}
\def\be{\beta}

\def\Om{\Omega}
\def\al{\alpha}

\def\la{{\lambda}}
\def\La{{\Lambda}}
\def\si{{\sigma}}

\newcommand{\de}{{\bf \delta}}
\def\Dist{\mathrm{Dist}}
\def\dist{\mathrm{d}}
\def\diam{\mathrm{diam}}
\def\calC{\mathcal{C}}
\def\se{{\subseteq}}
\def\d{\mathrm d}

\def\th{\theta}
\def\pa{{\partial}}
\def\len{{\mathrm{l}}}

\def\sm{\setminus}

\def\and{{\quad\text{and}\quad}}

\title{Collet, Eckmann and the bifurcation measure}

\author{Matthieu Astorg}
\address{MAPMO, 
Universit\'e d'Orl\'eans, Route de Chartres
B.P. 6759 - 45067 ORL\'EANS Cedex 2, FRANCE}
\email{matthieu.astorg@univ-orleans.fr}
  \author{Thomas Gauthier}
\address{LAMFA, UPJV, 33 rue Saint-Leu, 80039 AMIENS Cedex 1, FRANCE}
\address{CMLS, \'Ecole polytechnique, CNRS, Universit\'e Paris-Saclay, 91128 PALAISEAU Cedex, FRANCE}
\email{thomas.gauthier@u-picardie.fr}
  \author{Nicolae Mihalache}
\address{LAMA, UPEC, 61 avenue du G\'en\'eral de Gaulle, 94010 CR\'ETEIL Cedex, FRANCE}
\email{nicolae.mihalache@u-pec.fr}
\author{Gabriel Vigny}
\address{LAMFA, UPJV, 33 rue Saint-Leu, 80039 AMIENS Cedex 1, FRANCE}
\email{gabriel.vigny@u-picardie.fr}

\begin{document}

\begin{abstract}
The moduli space $\mathcal{M}_d$ of degree $d\geq2$ rational maps can naturally be endowed with a measure $\mu_\bif$ detecting maximal bifurcations, called the bifurcation measure. We prove that the support of the bifurcation measure $\mu_\bif$ has positive Lebesgue measure. To do so, we establish a  general sufficient condition for the conjugacy class of a rational map to belong to the support of $\mu_\bif$ and we exhibit a large set of Collet-Eckmann rational maps which satisfy that condition. As a consequence, we get a set of  Collet-Eckmann rational maps of positive Lebesgue measure which are approximated by hyperbolic rational maps.
\end{abstract}

\maketitle
\tableofcontents

\section{Introduction}
For an integer $d\geq 2$, we let $\rat_d$ be the space of rational maps of degree $d$ of the Riemann sphere $\P^1$. 
The space $\rat_d$ of degree $d$ rational maps is a quasi-projective subvariety of dimension $2d+1$. More precisely, there exists an (irreducible) variety $V$ such that $\rat_d\simeq\P^{2d+1}\setminus V$. The $J$-\emph{stability locus} in $\rat_d$ is defined as the set of maps $f$ which are structurally stable on their Julia set $\J_f$. By the seminal work~\cite{MSS} of Ma\~n\'e, Sad and Sullivan, it is an open and dense subset of $\rat_d$.
The {\em bifurcation locus} in $\rat_d$ is its complement.

One can give a measurable description of this bifurcation as follows: any rational map $f\in\rat_d$ has a unique measure $\mu_f$ of maximal entropy $\log d$. The \emph{Lyapunov exponent} of $f$ with respect to the measure $\mu_f$ can be defined as $L(f):=\int_{\P^1}\log|f'|\, \mu_f$, where $|\cdot|$ is any hermitian metric on $\P^1$.
 DeMarco \cite{DeMarco1} proved that the function $L:f\in\rat_d\mapsto L(f)\in\R_+$ is plurisubharmonic (\emph{p.s.h} for short) and continuous and that the bifurcation locus is the support of the positive closed $(1,1)$-current $T_{\bif}:= dd^c L$.

Moreover, M\"obius transformations act by conjugacy on $\rat_d$ and the quotient space is an orbifold affine variety of dimension $2d-2$ which is known as the {\em moduli space} $\mathcal{M}_d$ of degree $d$ rational maps. As the function $L$ is invariant by conjugacy, the function $L$ induces a \emph{p.s.h} and continuous function $\mathcal{L}:\mathcal{M}_d\longrightarrow\R$. The support of $dd^c\mathcal{L}$ coincides with the image of the bifurcation locus under the quotient map $\Pi:\rat_d\to\mathcal{M}_d$.
  \begin{definition}
The \emph{bifurcation measure} of the space $\mathcal{M}_d$ is the positive measure
\begin{center}
$\mu_\bif := (dd^c\mathcal{L})^{2d-2}$.
\end{center}
\end{definition}
Notice that the support of $\mu_\bif$ is strictly contained in the bifurcation locus.
This measure has been introduced by Bassanelli and Berteloot in \cite{BB1}. Its support is exactly the accumulation set of the hyperbolic (conjugacy class of) rational maps (see~\cite{BB1} and also Buff and Epstein \cite{buffepstein} and the second author \cite{neutral} for other dynamical characterizations). Recall that a rational map is \emph{hyperbolic} if it is uniformly expanding on its Julia set  $\J_f$, i.e. there exist $C>0$ and $\lambda>1$ such that for any $n\geq1$ and any $z\in\J_f$, $|(f^n)'(z)|\geq C\cdot \lambda^n$.
 Another way to define $\mu_\bif$ is as follows: it is the only finite measure on $\mathcal{M}_d$ such that $\Pi^*(\mu_\bif)=T_\bif^{2d-2}$ as closed positive $(2d-2,2d-2)$-currents on $\rat_d$ (see~\cite[\S 6]{BB1}).

~

In the family $\{z^2+c\}_{c\in\mathbb{C}}$ of quadratic polynomials, this measure is exactly the harmonic measure of the Mandelbrot set and its support is exactly the boundary of the Mandelbrot set. One of the important questions in complex dynamics which is still open is the following.
\begin{question}
Does the boundary of the Mandelbrot set have positive area?
\end{question}
It is natural to ask similar questions in the moduli space  $\mathcal{M}_d$ of degree $d$ rational maps.
For that, let us \emph{naturally} endow $\mathcal{M}_d$ with a volume form. the moduli space $\mathcal{M}_d$ is a affine variety in some $\P^N$. Let $\omega_{\P^N}$ be the Fubini-Study form on $\P^N$.  Its restriction defines a $(1,1)$ form $\omega_d$ on $\mathcal{M}_d$ and we consider the volume form on $\mathcal{M}_d$
\[\textup{Vol}_{\mathcal{M}_d}:=\omega_d^{2d-2}.\]
This is (up to renormalization) a probability measure on $\mathcal{M}_d$. Moreover, in any (smooth) local chart of $\mathcal{M}_d$ it is smoothly equivalent to the Lebesgue measure, hence it is a non-degenerate volume form in $\mathcal{M}_d$.
 For the bifurcation locus, Rees showed that it has positive $\textup{Vol}_{\mathcal{M}_d}$-measure \cite{Rees}.  This leaves the question of the volume of the support of the measure $\mu_\bif$ in $\mathcal{M}_d$.

 Another major problem concerning the moduli space $\mathcal{M}_d$ is the \emph{Hyperbolicity Conjecture}:
 \begin{conjecture}[Hyperbolicity]
 	Hyperbolic rational maps are dense in $\rat_d$.
 \end{conjecture}
 In the present paper, we are interested in the following simpler but related problem: is any Collet-Eckmann rational map approximated by hyperbolic rational maps? Refining Rees' results, Aspenberg \cite{Asp13} showed that there is a set of positive $\textup{Vol}_{\mathcal{M}_d}$-measure of (suitable) Collet-Eckmann rational map (such maps are in the bifurcation locus).
A rational map $f\in\rat_d$ is \emph{Collet-Eckmann} if the critical set $\mathcal{C}(f)$ of $f$ is contained in $\J_f$ and if there exist $\gamma,\gamma_0>0$ such that 
 \begin{equation}\label{eqn:CE}
 |(f^n)'(f(c))|\geq e^{n\gamma-\gamma_0},\tag{CE$(\gamma,\gamma_0)$}
 \end{equation}
 for any $c\in \mathcal{C}(f)$ and any $n\geq 0$.

Our main result is the following
\begin{Theorem}\label{tm:leb}
Pick an integer $d\geq2$ and let $[f]\in\supp(\mu_\bif)$. Then for any open neighborhood $\Omega\subset\mathcal{M}_d$ of $[f]$, there is a set $CE_\Omega \subset \Omega \cap \supp(\mu_\bif)$ of Collet-Eckmann maps with:
\[\textup{Vol}_{\mathcal{M}_d}(CE_\Omega ))>0.\]
\end{Theorem}
This implies the following result reated to the above conjecture.
\begin{Corollary}\label{tm:approxhyperbolic}
There exists a set of positive $\textup{Vol}_{\rat_d}$-measure of Collet-Eckmann rational maps $f$ of $\rat_d$ with $[f] \in \supp(\mu_\bif)$ such that $f$ is approximated by hyperbolic rational maps.
\end{Corollary} 
We rely on the following strategy: we first give a very general sufficient condition for a conjugacy class of rational maps to belong to the support of the bifurcation measure (see Theorem~\ref{tm:suppTm}). Then we exhibit a large set of rational maps fulfilling this condition. Theorem~\ref{tm:leb} becomes a corollary of Theorem~\ref{tm:suppTm}, Theorem~\ref{tm:abundance} and Theorem~\ref{tm:bigscale} below and of the Main Theorem of~\cite{buffepstein}.

~

When a rational map has simple critical points, one can follow them holomorphically as maps $c_1,\ldots,c_{2d-2}:V\to\P^1$ where $V$ is an open neighborhood of $f$ in $\rat_d$. For any $(2d-2)$-tuple of positive integers $\un:=(n_1,\ldots,n_{2d-2})$, one also can define a holomorphic map $\mathfrak{V}_{\un}:V\longrightarrow(\P^1)^{2d-2}$ by letting
\[\mathfrak{V}_{\un}(g):=\left(g^{n_1}(c_{1}(g)),\ldots,g^{n_{2d-2}}(c_{2d-2}(g))\right).\]
This map detects, in a certain sense, the collective stability/instability of the critical points $c_1,\ldots,c_{2d-2}$. In the sequel, we will use the following definition in a crucial way.
\begin{definition}
Let $f\in\rat_d$ with simple critical values and let $\Gamma_{\un}$ denote the graph of $\mathfrak{V}_{\un}$. We say that $f$ satisfies the \emph{large scale condition} if
 there exist a (local) complex submanifold $\Lambda_f\subset\rat_d$ of dimension $2d-2$, a sequence $\un_k=(n_{k,1},\ldots,n_{k,2d-2})$ of $(2d-2)$-tuples with $n_{k,j}\rightarrow+\infty$ for all $j$, a basis of neighborhoods $\{\Omega_k\}_{k\geq1}$ of $f$ in $\Lambda_f$ and $\delta>0$ such that for any $k$,  the connected component of $\Gamma_{\un_k}\cap \Omega_k\times \left(\D(f^{n_{k,1}}(c_{1}(f)),\delta)\times\cdots\times \D(f^{n_{k,2d-2}}(c_{{2d-2}}(f)),\delta) \right) $ containing $(f,\mathfrak{V}_{\un_k}(f))$ is contained in $\Omega'_k\times \left(\D(f^{n_{k,1}}(c_{1}(f)),\delta)\times\cdots\times \D(f^{n_{k,2d-2}}(c_{{2d-2}}(f)),\delta) \right)$ for some $\Omega'_k \Subset \Omega_k$.
\end{definition}
This means that a rational map $f$ satisfies the large scale condition if, infinitely many times, the map $\mathfrak{V}_{\un}$ sends an arbitrarily small neighborhood of $[f]$ in the moduli space $\mathcal{M}_d$ to a polydisk of fixed size in $(\P^1)^{2d-2}$ and its graph is vertical-like near $[f]$.  
The first step of our proof consists in proving that the large scale condition is actually a sufficient condition for maximal bifurcations to occur at $f$. More precisely, we prove the following.
\begin{Theorem}\label{tm:suppTm}
Pick $f\in\rat_d$. Assume that $\omega(c)\subset\J_f$ for all $c\in\mathcal{C}(f)$, that $f$ has simple critical points and that $f$ satisfies the large scale condition. Then $[f]\in\supp(\mu_\bif)$.
\end{Theorem}
 This condition extracts the core of the ones previously used, such as having a uniform expansion along the postcritical set, the famous Misiurewicz condition. To illustrate the interest of this condition, we show in \S\ref{sec:Mis} that this condition is satisfied by Misiurewicz maps. The proof drastically simplifies the proof of~\cite[Theorem~1.4]{Article1}. In particular, we now are able to avoid any linearization process along the critical orbits which was a crucial step in the proof given by the second author (see~\cite[\S 5]{Article1}). Indeed, we only use here the transversality and the holomorphic motion of the hyperbolic set containing the post-critical set.
The proof of Theorem~\ref{tm:suppTm} is based on a phase-parameter transfer (giving a measurable version of Tan Lei's work \cite{Tan}). For clarity, we give the proof in the easier case of one critical point in Section~\ref{toy-model}.

~

% Relying on results of Aspenberg~\cite{Asp13}, we then show the abundance of classes of Collet-Eckmann rational maps in any open set $U\subset\mathcal{M}_d$ with $U\cap\supp(\mu_\bif)\neq\varnothing$. 
A rational map is \emph{strongly Misiurewicz} if all its critical points are preperiodic to repelling cycles. Reformulating the main result of Aspenberg~\cite{Asp13} in terms of the conditions CE($\gamma,\gamma_0$), CE2($\mu,\mu_0$), BA($\alpha$) and  FA($\eta, \iota$) (defined at the beginning of Section~\ref{sec:distortion}) gives the following.

\begin{Theorem}\label{tm:abundance}
	Assume that $f\in\rat_d$ is strongly Misiurewicz, has simple critical points and is not a flexible Latt\`es map. 
	Then, there exist $\mu,\mu_0,\gamma,\gamma_0>0$ and $\hat{\alpha}>0$ such that for all $\alpha < \min(\frac{\gamma}{200}, \hat \alpha)$, 
	there exist $\hat{\eta}>0$ and $\hat{\iota}>0$ such that for all $\eta < \hat \eta$ and for all
	$\iota < \hat \iota$, the map $f$ is a Lebesgue density point of rational maps satisfying
	CE($\gamma,\gamma_0$), CE2($\mu,\mu_0$), BA($\alpha$) and  FA($\eta, \iota$).
\end{Theorem}

The last key ingredient to prove Theorem~\ref{tm:leb} can be formulated as follows (conditions (K1-6) are defined in Section~\ref{sec51}).
\begin{Theorem}\label{tm:bigscale}
Let $\gamma, \gamma_0, \mu, \mu_0,\eta,\kappa>0$ and $\alpha<\gamma/200$.
There exists $\iota>0$ such that any $f\in\rat_d$ with simple critical points and
satisfying CE($\gamma,\gamma_0$), CE2($\mu,\mu_0$), BA($\alpha$),  FA($\eta, \iota$) and (K1-6)
 satisfies the large scale condition.
\end{Theorem}
To prove Theorem~\ref{tm:bigscale}, we follow Tsujii's generalization \cite{Tsu} of Benedicks and Carleson construction \cite{BenCar} that we need to adapt to the complex setting. The strategy of the proof is summarized below.
\renewcommand{\labelitemi}{•}
\def\crit{\mathcal{C}}
\begin{itemize}
	\item Take a Collet-Eckmann map $f$ and a small ball in the dynamical plane centered at a critical value. That ball will go to the large scale with exponential growth and good distortion estimates under $f^n$ by (CE) as long as its orbit stays far away from the critical set $\crit(f)$. For each $n$ we will choose such a starting ball $B_n$. 
	\item Passages of the critical orbit near $\crit(f)$ impose upper bounds of the size of $B_n$. The assertion (BA) gives lower bounds for the approach rate to $\crit(f)$. After a close visit near $\crit(f)$, the image is even closer to the visited critical value. \rlem{Kx} shows that the sequel of the critical orbit copies the good properties of a long prefix of the orbit of the visited critical value.
	\item Lemma~\ref{lemBkDil} guarantees that just before going (again) near a critical point, we gain expansivity by (CE2), so we have restored the exponential growth and the bound for the distortion on $B_n$. We use a Lemma à la M\~ané (Proposition~\ref{propMane}) for the suffix of a (finite) critical orbit that does not visit a neighborhood of $\crit(f)$.
	 \item Lemma~\ref{lemAplus} gives a large density of times $n$ for which $B_n$ goes to the large scale ((FA) tells that the critical orbit does not go too often near the critical set). We intersect over all the critical values and still have positive density of times for which the starting balls go to the large scale for all the critical values.
	\item Using Lemma~\ref{lemParDist} allows to bound parametric distortion on complex lines passing through $f$ thanks to the transversality of critical relations (using again (CE) and (BA)).
	\item Finally, we use a result of Sibony and Wong \cite{SW} and the transversality to extend the distortion to a ball on a neighborhood of $f$ to get the large scale condition (Theorem~\ref{theo_largescale}).
\end{itemize}

\paragraph*{Structure of the paper}
In Section~\ref{sec:currents}, we recall facts on bifurcation currents. In Section~\ref{sec:TM}, we define a (generalized) large scale condition and show that if a parameter satisfies it for some $m$ then it is in the support of $T_\bif^m$. In Section~\ref{sec:Mis}, we apply this result to $k$-Misiurewicz maps. In Section~\ref{sec:distortion}, we prove Theorems \ref{tm:abundance} and \ref{tm:bigscale}. In Section~\ref{sec:mainresults}, we prove Theorem~\ref{tm:leb} and a strengthened version of Corollary~\ref{tm:approxhyperbolic}.

\section{Basics on bifurcation currents}\label{sec:currents}
 This section is devoted to the bifurcation currents. We begin with giving a description of the bifurcation currents. We then  give a new formula for the higher bifurcation currents.

\subsection{The bifurcation currents of critical points}\label{sec:bifcur}
Let $(f_\lambda)_{\lambda\in\Lambda}$ be a holomorphic family of rational maps equipped with $2d-2$ marked critical points $c_1,\ldots,c_{2d-2}:\Lambda\to\P^1$.
\begin{definition}
A critical point $c$ is said to be \emph{marked} if there exists a holomorphic function $c:\Lambda\longrightarrow\P^1$ satisfying $f_\lambda'(c(\lambda))=0$ for every $\lambda\in \Lambda$.
\par We say that the critical point $c$ is \emph{active} at $\lambda_0\in \Lambda$ if $(f_\lambda^n(c(\lambda)))_{n\geq0}$ is not a normal family in any neighborhood of $\lambda_0$. Otherwise we say that $c$ is \emph{passive} at $\lambda_0$. The \emph{activity locus} of $c$ is the set of parameters $\lambda\in \Lambda$ at which $c$ is active.\label{actif}
\end{definition}

Let us construct the bifurcation current of the critical point $c_i$. One can define a fibered dynamical system $\hat f$ acting on $\Lambda\times\P^1$
\begin{eqnarray*}
\hat f:\Lambda\times\P^1 & \longrightarrow & \Lambda\times\P^1\\
(\lambda,z) & \longmapsto & (\lambda,f_\lambda(z))~.
\end{eqnarray*}
We denote by $p_\Lambda:\Lambda\times\P^1\to\Lambda$ and $p_{\P^1}:\Lambda\times\P^1\to\Lambda$ the respective natural projections and by $\hat\omega:=(p_{\P^1})^*\omega_{\textup{FS}}$, where $\omega_{\textup{FS}}$ is the Fubini-Study form on $\P^1$ normalized so that $\int_{\P^1}\omega_{\textup{FS}}=1$.
 We say that a function $\psi$ is $\hat\omega$-psh if it can be locally written as the sum of a smooth function and a plurisubharmonic function (psh for short) and $dd^c \psi+ \hat\omega \geq 0$ in the sense of currents. Then, there exists a $\hat\omega$-psh function $g$ such that 
\begin{eqnarray}
(\hat f)^*\left(\hat\omega+dd^cg\right)=d\cdot \left(\hat\omega+dd^cg\right).\label{eq:pullback}
\end{eqnarray}
Indeed, since $d^{-1}(\hat f)^*\hat\omega$ and $\hat\omega$ are in the same cohomology class, there exists a smooth $\hat\omega$-psh function $u$ such that $d^{-1}(\hat f)^*\hat\omega=\hat\omega+dd^cu$. Taking
$$u_n:=\sum_{j=0}^{n-1}\frac{u\circ (\hat f)^j}{d^j}~,$$
we defined $g:=\lim_nu_n$. The function $g$ is continuous and $\hat\omega$-plurisubharmonic on $\Lambda\times\P^1$, since $\|g-u_n\|_\infty=O(d^{-n})$. The function $g$ is the \emph{Green function} of $\hat f$ and is unique up to an additive constant. We shall use the following notation in the sequel
\[\widehat{T}:=\hat\omega+dd^cg~.\]

One can give the following definition.

\begin{definition}
The \emph{bifurcation current} of the critical point $c_i$ in $(f_\lambda)_{\lambda\in\Lambda}$ is
$$T_i:= (p_\Lambda)_*\left((\hat\omega +dd^cg)\wedge [\hat V_i]\right)~,$$
where $\hat V_i=\{(\lambda,v_i(\lambda))\ : \lambda\in\Lambda\}$ is the graph of the map $v_i(\lambda):=f_\lambda(c_i(\lambda))$.\label{def:Tc}
\end{definition}

The holomorphic family $(f_\lambda)_{\lambda\in\Lambda}$ admits local lifts, i.e. for any small enough $V\subset\Lambda$, there exists a holomorphic family $(F_\lambda)_{\lambda\in V}$ of non-degenerate homogeneous degree $d$ polynomial endomorphims of $\C^2$, and the \emph{Green function} of the lift is then
$$G(\lambda,x,y):=\lim_{n\to\infty}d^{-n}\log\|F_\lambda^n(x,y)\|~, \ (x,y)\in\C^2\setminus\{0\}~.$$
It actually is a continuous and psh function on $V\times(\C^2\setminus\{0\})$. According to \cite[Section 3.2.2]{bsurvey}, one can prove that for any local holomorphic section $\sigma:U\subset\P^1\to\C^2\setminus\{0\}$ of the canonical projection $\C^2\setminus\{0\}\to\P^1$, then
\begin{center}
$\hat\omega+dd^cg=dd^cG(\lambda,\sigma(z))~, \ \text{on} \ V\times U.$
\end{center}
We thus can locally write $T_i=dd^cG(\lambda,\sigma\circ v_i(\lambda))=d\cdot dd^cG(\lambda,\sigma\circ c_i(\lambda))$ on $U$.

\begin{remark} \normalfont
The definition we give here is not the classical one. The usual definition of the bifurcation currents is to take locally $\tilde T_i:=dd^cG(\lambda,\sigma\circ c_i(\lambda))=d^{-1}\cdot T_i$, which does not change the support of the current.
\end{remark}
\noindent The important information concerning the current $T_i$ is the following (see \cite{favredujardin}).

\begin{proposition}[Dujardin-Favre]
The support of $T_i$ is the activity locus of $c_i$.
\end{proposition}

Another way to characterize the bifurcation current of $c_i$ is the following. Let $\xi_{n}^{i}:\Lambda\to\P^1$ be the map given by $\xi_{n}^{i}(\lambda):=f_\lambda^n(v_i(\lambda))$, for $n\geq0$ and $1\leq i\leq 2d-2$. The sequence of forms $d^{-n}(\xi_{n}^{i})^*\omega_{\textup{FS}}$ converge to the current $T_i$ in the sense of currents (see e.g. \cite{higher}).

\subsection{Bifurcation currents of a holomorphic family}
Recall that $f\in\rat_d$ admits a unique maximal entropy measure $\mu_f$. The \emph{Lyapounov exponent} of $f$ with respect to the measure $\mu_f$ is the real number $L(f)\pe\int_{\P^1}\log|f'|\mu_f$. For a holomorphic family $(f_\lambda)_{\lambda\in \Lambda}$ of degree $d$ rational maps, we denote by $L(\lambda)\textup{:=} L(f_\lambda)$. Then, the function $\lambda \longmapsto L(\lambda)$ is called the \emph{Lyapounov function} of the family $(f_\lambda)_{\lambda\in \Lambda}$. It is a psh and continuous function on $\Lambda$ (see \cite{BB1} Corollary 3.4). The Margulis-Ruelle inequality implies that $L(f)\geq\frac{\log d}{2}$.

\par When $(f_\lambda)_{\lambda\in \Lambda}$ is with $2d-2$ distinct marked critical points $c_1,\ldots,c_{2d-2}$, the bifurcation locus in the sense of Ma\~n\'e-Sad-Sullivan and Lyubich (see \cite{lyubich-stable,MSS}) coincides with the union of the activity loci of the $c_i$'s. According to DeMarco \cite{DeMarco1}, we have the following.

\begin{theorem}[DeMarco]
Let $(f_\lambda)_{\lambda\in\Lambda}$ be a holomorphic family of degree $d$ rational maps with $2d-2$ distinct marked critical points. Then
the current $T_\bif:=dd^cL$ is exactly supported by the bifurcation locus. Moreover, $T_\bif=d^{-1}\sum_{i=1}^{2d-2}T_i$.
\end{theorem}

The current $T_\bif$ is the \emph{bifurcation current} of of the family $(f_\lambda)_{\lambda\in\Lambda}$. The self-intersections of the current $T_\bif$ have been first studied by Bassanelli and Berteloot \cite{BB1}. 
\begin{definition}
 We define the $m^{\text{th}}$-\emph{bifurcation current} of the family $(f_\lambda)_{\lambda\in\Lambda}$ by setting
$$T_\bif^m:=\bigwedge_{i=1}^mT_\bif~.$$
\end{definition}
It is known that for all $1\leq i\leq 2d-2$, we have $T_i\wedge T_i=0$ (see \cite[Theorem 6.1]{Article1}) and one can show that
\begin{eqnarray}
T_\bif^m=m!\sum_{i_1<\ldots<i_m}T_{i_1}\wedge\cdots \wedge T_{i_m}~.\label{eq:DeM}
\end{eqnarray}

\subsection{A formula for higher bifurcation currents}
We want here to give a similar expression as the one given in Definition~\ref{def:Tc} for the higher bifurcation current associated to a $m$-tuple of critical points. Let us introduce some notations. Let $(f_\lambda)_{\lambda\in\Lambda}$ be a holomorphic family of degree $d$ rational maps with $m$ marked critical points, $c_1,\ldots,c_m:\Lambda\to\P^1$, with $1\leq m\leq \min(2d-2,\dim\Lambda)$. 
As above, we define $v_j:\Lambda\to\P^1$ for $\lambda\in\Lambda$ by $v_j(\lambda):=f_\lambda(c_j(\lambda))$.
This time, we let 
\[\mathcal{V}_{j}:=\{(\lambda,z)\in\Lambda\times(\P^1)^m \ : \ z_j=v_{j}(\lambda)\}\]
We finally let $\pi_\Lambda:\Lambda\times(\P^1)^m\longrightarrow\Lambda$ and, for $1\leq j\leq m$, we let
\begin{eqnarray*}
\pi_j:\Lambda\times(\P^1)^m & \longrightarrow & \Lambda\times\P^1\\
(\lambda, z) & \longmapsto & (\lambda,z_j)
\end{eqnarray*}
be the respective natural projection. Let $T_i$ be the bifurcation current of $c_i$ in $(f_\lambda)_{\lambda\in\Lambda}$.

\begin{lemma}\label{lm:formula}
With the above notations, we have
\begin{center}
$\displaystyle\bigwedge_{j=1}^mT_{j}=
(\pi_\Lambda)_*\left(\bigwedge_{j=1}^m(\pi_j)^*\left(\widehat{T}\right)\wedge[\mathcal{V}_{j}]\right)~.$
\end{center}
\end{lemma}

\begin{proof}
It is a local problem, so we can assume that $\Lambda=\B$ is a ball of $\C^N$ for some $N\geq1$. Recall that, up to reducing the ball, one can also write $\widehat{T}=dd^c_{\lambda,z}G_\lambda(\sigma(z))$, where $G_\lambda$ is the Green function of a holomorphic family of non-degenerate homogeneous polynomial lift $F_\lambda$ of $f_\lambda$, i.e.
$$G_\lambda(z_1,z_2):=\lim_{n\to\infty}d^{-n}\log\|F^n_\lambda(z_1,z_2)\|, \ (z_1,z_2)\in\C^2\setminus\{0\}~,$$
and $\sigma$ is any local section of the natural projection $\C^2\setminus\{0\}\to\P^1$. Let $\sigma_j$ be such a section which up to reducing the ball contains the image of the map $v_{j}:\B\to\P^1$, then $dd^cG_\lambda(\sigma_j\circ v_{j}(\lambda))=T_{j}$ (see e.g.~\cite[Section 3.2.2]{bsurvey}).

~

Let now $\phi$ be a smooth test $(N-m,N-m)-$form on $\B$ and let $\mathcal{V}:=\bigcap_j\mathcal{V}_{j}$.  Let $p(\lambda):=(\lambda,v_1(\lambda),\dots,v_{m}(\lambda))$, so that $\pi_\Lambda\circ p=\textup{id}_\B$ and $p(\B)=\mathcal{V}$. Then
\begin{align*}
\left\langle (\pi_\Lambda)_*\left(\bigwedge_{j=1}^m\left((\pi_j)^*\widehat{T}\right)\wedge[\mathcal{V}_{j}]\right),\phi\right\rangle & = \int \bigwedge_{j=1}^m(\pi_j)^*\left(\widehat{T}\right)\wedge[\mathcal{V}_{j}]\wedge(\pi_\Lambda)^*\phi \\
& = \int_{\mathcal{V}} \bigwedge_{j=1}^m(\pi_j)^*\left(\widehat{T}\right)\wedge(\pi_\Lambda)^*\phi \\
& = \int_{\mathcal{V}} \bigwedge_{j=1}^m(\pi_j)^*dd^c_{\lambda,z}(G_\lambda \circ \sigma_j)\wedge(\pi_\Lambda)^*\phi \\
& = \int_{p(\B)} \bigwedge_{j=1}^mdd^c_{\lambda,z}G_\lambda(\sigma_j(z_j))\wedge(\pi_\Lambda)^*\phi \\
& = \int_{\B}\bigwedge_{j=1}^mdd^cG_\lambda(\sigma_j\circ v_{j}(\lambda))\wedge\left(\left(p^*(\pi_\Lambda)^*\phi\right)\right)\\
& = \left\langle \bigwedge_{j=1}^mT_{j},\phi\right\rangle~.
\end{align*}
\end{proof}

\section{Generalized large scale condition and the bifurcation currents}\label{sec:TM}
For the whole section, we let $(f_\lambda)_{\lambda\in\Lambda}$ be a holomorphic family of rational maps with $m\geq1$ marked critical points $c_1,\ldots,c_m:\Lambda\longrightarrow\P^1$. As above, we use the notation $v_{j}(\lambda):=f_\lambda(c_{j}(\lambda))$.
We set $\mathfrak{c}:=(c_1,\ldots,c_m)$ and we also use the following notation: For any $m$-tuple of positive integers $\un=(n_1,\ldots,n_m)$, any $1\leq j\leq m$ and any $\lambda\in\Lambda$, we let 
\[\xi_{n_j}^{j}(\lambda)\pe f_\lambda^{n_j}(v_{j}(\lambda))~, \ \text{and} \ 
\mathfrak{V}_{\un}^{\mathfrak{c}}(\lambda):=\left(\xi_{n_1}^{1}(\lambda),\ldots,\xi_{n_m}^{m}(\lambda)\right)~.\]
This way, we define a holomorphic map $\mathfrak{V}_{\un}^{\mathfrak{c}}:\Lambda\longrightarrow(\P^1)^m$.
 We denote by $\mathcal{V}_{\un}$ the graph of $\mathfrak{V}_{\un}^{\mathfrak{c}}$. 
~

We now can define the (generalized) large scale condition.
\begin{definition}[Generalized large scale condition]
We say that a parameter $\lambda_0\in\Lambda$ satisfies the \emph{generalized large scale condition at} $f_{\lambda_0}$ for the $m$-tuple $(c_1,\ldots,c_m)$ in $\Lambda$ if there exist a sequence $\un_k=(n_{k,1},\ldots,n_{k,m})$ of m-tuples with $n_{k,j}\rightarrow+\infty$ and a basis of neighborhood $\{\Omega_k\}_{k\geq1}$ of $f$ in $\Lambda$ and $\delta>0$ such that for any $k$, the connected component of $\mathcal{V}_{\un_k}\cap \Omega_k\times \left(\D(\xi_{n_{k,1}}^1(\lambda_0),\delta)\times\cdots\times \D(\xi_{n_{k,m}}^m(\lambda_0),\delta) \right) $ containing $(\lambda_0,\mathfrak{V}_{\un_k}^{\mathfrak{c}}(\lambda_0))$ is contained in $\Omega'_k\times \left(\D(\xi_{n_{k,1}}^1(\lambda_0),\delta)\times\cdots\times \D(\xi_{n_{k,m}}^m(\lambda_0),\delta) \right)$ for some $\Omega'_k\Subset \Omega_{k}$. 
\end{definition}
We prove in this section the following result which we view as a general sufficient condition for a parameter to belong to the support of a (higher) bifurcation current.

\begin{theorem}\label{tm:sumplus}
Let $1\leq m\leq \min(2d-2,\dim\Lambda)$ be an integer and let $\lambda_0\in\Lambda$. Assume that $\lambda_0$ satisfies the generalized large scale condition for $\mathfrak{c}:=(c_1,\ldots,c_m)$ in a local submanifold $S\ni \lambda_0$ of $\Lambda$ with $\dim S=m$. Assume in addition that $\omega(c_i(\lambda_0))\subset\J_{\lambda_0}$ for all $1\leq i\leq m$. Then $\lambda_0\in\supp(T_{1}\wedge\cdots\wedge T_{m})$.
\end{theorem}

\begin{proof}[Proof of Theorem~\ref{tm:suppTm}]
By definition of $\mu_\bif$, we have $[f]\in\supp(\mu_\bif)$ if and only if $f\in\supp(T_\bif^{2d-2})$. On the other hand, since $f$ has simple critical points, there exist a neighborhood $U\subset\rat_d$ of $f$ and holomorphic maps $c_1,\ldots,c_{2d-2}:U\rightarrow\P^1$ with $\mathcal{C}(g)=\{c_1(g),\ldots,c_{2d-2}(g)\}$ for all $g\in U$. We then can apply the above Theorem~\ref{tm:sumplus} to $\mathfrak{c}=(c_1,\ldots,c_{2d-2})$. Since, by~\eqref{eq:DeM}, $T_\bif^{2d-2}=(2d-2)!T_1\wedge\cdots\wedge T_{2d-2}$ on $U$, the result follows.
\end{proof}

\begin{remark}\normalfont
Note that the assumption $\omega(c_j(\lambda_0))\subset\J_{\lambda_0}$ is satisfied not only when $c_j(\lambda_0)\in\J_{\lambda_0}$, but also when $c_j(\lambda_0)$ belongs to a parabolic basin.
\end{remark}

First, note that it is sufficient to treat the case when $\dim\Lambda=m$ and $\lambda_0$ satisfies the large scale condition for $\mathfrak{c}:=(c_1,\ldots,c_m)$ in a $\Lambda$. Indeed, by \cite[Lemma 6.3]{Article1}, if $\lambda_0\in\supp(\left(T_{1}\wedge\cdots\wedge T_{m}\right)|_S)$, then $\lambda_0\in\supp\left(T_{1}\wedge\cdots\wedge T_{m}\right)$. We hence may assume $S=\Lambda$ has dimension $m$ and let
\[\mu:=T_{1}\wedge\cdots\wedge T_{m}~.\]
It defines a positive measure on $\Lambda$.

~

Let $(\lambda_1,\ldots,\lambda_m)$ be a local system of holomorphic coordinates centered at $\lambda_0$. We let $\D_\delta^m$ be the polydisk of radius $\delta>0$ of $\Lambda$ centered at $\lambda_0$ in those coordinates. 

\subsection{The case $m=1$: a toy-model for the general case}\label{toy-model}

We give here the proof of Theorem~\ref{tm:sumplus} in the case $m=1$. We let $c:\Lambda\to\P^1$ be the marked critical point satisfying the large scale condition at $\lambda_0$. As above, for $n\geq0$ and $\lambda\in\Lambda$, write
\[\xi_{n}(\lambda):=f_\lambda^{n+1}(c(\lambda))=f_\lambda^{n}(\xi_0(\lambda))~.\]
 Of course, in that case, it is easy to see that the large scale condition implies the non-normality of the family $(f^n_\lambda(c(\lambda)))_n$ but we provide here a proof that can be adapted to work with higher degree bifurcation currents. For that, recall that $\mathcal{V}_{n}$ is the graph of $\xi_{n}:\Lambda\to\P^1$. Recall that $d^{-1}(\hat{f})^*\widehat{T}=\widehat{T}$. Choose any $\epsilon>0$. We shall prove that $\mu(\D(\lambda_0,\epsilon))>0$ for all $\epsilon>0$ small enough. For any $n\geq1$
\begin{eqnarray*}
\mu\left(\D(\lambda_0,\epsilon)\right) & = & \int_{\D(\lambda_0,\epsilon)\times\P^1}\widehat{T}\wedge[\mathcal{V}_{0}]\\
& = & d^{-n}\int_{\D(\lambda_0,\epsilon)\times\P^1}(\hat{f}^{n})^*\left(\widehat{T}\right)\wedge[\mathcal{V}_{0}]~.
\end{eqnarray*}
As a consequence,
\begin{eqnarray*}
I_n & := & d^{n}\mu\left(\D(\lambda_0,\epsilon)\right)=\int_{\D(\lambda_0,\epsilon)\times\P^1}(\hat{f}^{n})^*\left(\widehat{T}\right)\wedge[\mathcal{V}_{0}]\\
& = & \int_{\D(\lambda_0,\epsilon)\times\P^1}\widehat{T}\wedge(\hat{f}^{n})_*[\mathcal{V}_{0}]=\int_{\D(\lambda_0,\epsilon)\times\P^1}\widehat{T}\wedge[\mathcal{V}_{n}],
\end{eqnarray*}
since on one hand $\hat{f}^{-1}(\D(\lambda_0,\epsilon)\times\P^1)=\D(\lambda_0,\epsilon)\times\P^1=\hat{f}\left(\D(\lambda_0,\epsilon)\times\P^1\right)$ and, on the other hand, $(\hat{f}^{n})_*[\mathcal{V}_{0}]=[\mathcal{V}_{n}]$.
Let now $\delta>0$, $(n_k)$ and $(\Omega_k)$ be given by the large scale condition. Up to extraction, we can assume $\xi_{n_k}:=\xi_{n_k}(\lambda_0)$ converges to some point $x\in\J_{\lambda_0}$. As a consequence, there exists $k_0\geq1$ such that $\Omega_k\subset\D(\lambda_0,\epsilon)$ and
\[\xi_{n_k}\left(\Omega_k\right)\supset\D\left(x,\delta/2\right)~,\]
for any $k\geq k_0$. We now let $S_k$ be the connected component of $\mathcal{V}_{n_k}\cap \Omega_{k}\times\D(x,\delta/2)$ containing $(0,\xi_{n_k})$. Then $S_k\subset\Omega_k\times\D(x,\delta/2)$, $[S_k]/\|[S_k]\|$ has mass $1$ and
\[[\mathcal{V}_{n_k}]\geq\mathbf{1}_{\D(\lambda_0,\epsilon)\times\D(x,\delta/2)}[\Gamma_{n_k}]\geq [S_k]~.\]
 for all $k\geq k_0$. Let now $S$ be any weak limit of the sequence $[S_k]/\|[S_k]\|$. Then $\supp(S)=\{\lambda_0\}\times\D(x,\delta/2)$ and $S$ has mass $1$ by the large scale condition. By extremality of the current $[\{\lambda_0\}\times\D(x,\delta/2)]$, we deduce $S=M\cdot[\{\lambda_0\}\times\D(x,\delta/2)]$, where $M^{-1}>0$ is the Fubini-Study area of $\D(x,\delta/2)$. As $\widehat{T}$ has continuous potential, $\widehat{T}\wedge [S_k]/\|[S_k]\|\to \widehat{T}\wedge S$ as $k\to\infty$  in the sense of measures so:
\begin{align*}
\liminf_{k\to\infty}\|[S_k]\|^{-1}\cdot I_{n_k} \geq\liminf_{k\to\infty}\int\widehat{T}\wedge\frac{[S_k]}{\|[S_k]\|}\geq \int\widehat{T}\wedge S&=M\cdot \int\widehat{T}\wedge[\{\lambda_0\}\times\D(x,\delta/2)]\\
                       &=M\cdot \mu_{\lambda_0}\left(\D(x,\delta/2)\right)
\end{align*}
 as $\widehat{T}|_{\lambda=0}=\mu_{\lambda_0}$. Since $x\in\J_f=\textup{supp}(\mu_{\lambda_0})$ then $\mu_{\lambda_0}\left(\D(x,\delta/2)\right)>0$, we get
\[\liminf_{k\rightarrow\infty} \|[S_k]\|^{-1}\cdot I_{n_k}>0~,\]
which means that $I_{n_k}>0$ for $k$ large enough so $\mu(\D(\lambda_0,\epsilon))>0$.
Since this holds for all $\epsilon>0$, this ends the proof.

\subsection{First step: pulling-back by a fibered dynamical system}

Let us define a family of fibered dynamical systems acting on $\Lambda\times(\P^1)^m$ as follows: for any $m$-tuple $\un:=(n_1,\ldots,n_m)\in(\mathbb{N}^*)^m$, we let
\begin{eqnarray*}
F_{\un}:\Lambda\times(\P^1)^m & \longrightarrow & \Lambda\times(\P^1)^m\\
(\lambda,z_1,\ldots,z_m) & \longmapsto & (\lambda,f^{n_1}_\lambda(z_1),\ldots,f^{n_m}_\lambda(z_m))~.
\end{eqnarray*}
For a $m$-tuple $\un=(n_1,\ldots,n_m)$ of positive integers, we also set 
\[|\un|:=n_1+\cdots+n_m~.\]
 Let us first partially rewrite the mass of $\mu$ on any open set in terms of iterated pull-back by one the $F_{\un}$'s.

\begin{lemma}\label{lm:formula2}
For any $m$-tuple $\un=(n_1,\ldots,n_m)$ of positive integers, we let $\mathcal{V}_{\un}$ be the graph in $\Lambda\times(\P^1)^m$ of $\mathfrak{V}_{\un}$. Then, for any Borel set $\Omega\subset\Lambda$, we have
\begin{eqnarray*}
\mu(\Omega) & = & d^{-|\un|}\int_{\Omega\times(\P^1)^m}\left(\bigwedge_{j=1}^m(\pi_j)^*\left(\widehat{T}\right)\right)\wedge\left[\mathcal{V}_{\un}\right]~.
\end{eqnarray*}
\end{lemma}

\begin{proof}
Recall that we defined in Section \ref{sec:bifcur} a dynamical system $\hat f$ acting on $\Lambda\times\P^1$ by
\begin{eqnarray*}
\hat f:\Lambda\times\P^1 & \longrightarrow & \Lambda\times\P^1\\
(\lambda,z) & \longmapsto & (\lambda,f_\lambda(z))~.
\end{eqnarray*}
By definition of $F_{\un}$ and $\hat f$, for all $j$, the following diagram commutes
\begin{center}
$\xymatrix {\relax
\Lambda\times(\P^1)^m \ar[r]^{F_{\un}} \ar[d]_{\pi_j} & \Lambda\times(\P^1)^m\ar[d]^{\pi_j} \\
\Lambda\times\P^1 \ar[r]_{\hat f^{n_j}} & \Lambda\times\P^1~.}$
\end{center}
In particular, by \eqref{eq:pullback}, we get
\begin{eqnarray*}
d^{-n_j}(F_{\un})^*\left((\pi_j)^*\left(\widehat{T}\right)\right) & = & d^{-n_j}(\pi_j\circ F_{\un})^*\left(\widehat{T}\right)= d^{-n_j}(\hat f^{n_j}\circ \pi_j)^*\left(\widehat{T}\right)\\
& = & d^{-n_j}(\pi_j)^*\left((\hat f^{n_j})^*\left(\widehat{T}\right)\right)= (\pi_j)^*\left(\widehat{T}\right).
\end{eqnarray*}
According to Lemma \ref{lm:formula}, the change of variable formula gives
\begin{eqnarray*}
\mu(\Omega) & = & 
\int_{\Omega}(\pi_{\Lambda})_*\left(\bigwedge_{j=1}^m(\pi_j)^*\left(\widehat{T}\right)\wedge[\mathcal{V}_{j}]\right)\\
& = & \int_{\pi_{\Lambda}^{-1}(\Omega)}\bigwedge_{j=1}^m(\pi_j)^*\left(\widehat{T}\right)\wedge[\mathcal{V}_{j}]\\
& = & d^{-|\un|}\int_{\Omega\times(\P^1)^m}(F_{\un})^*\left(\bigwedge_{j=1}^m(\pi_j)^*\left(\widehat{T}\right)\right)\wedge\left[\bigcap_{j=1}^m\mathcal{V}_{j}\right]
\end{eqnarray*}
where the last equality comes from $\pi_{\Lambda}^{-1}(\Omega)=\Omega\times(\P^1)^m$. Whence
\begin{eqnarray*}
\mu(\Omega) = d^{-|\un|}\int_{\Omega\times(\P^1)^m}\left(\bigwedge_{j=1}^m(\pi_j)^*\left(\widehat{T}\right)\right)\wedge(F_{\un})_*\left[\bigcap_{j=1}^m\mathcal{V}_{j}\right]
\end{eqnarray*}
where we used $(F_{\un})_*\bigwedge_j[\mathcal{V}_{j}]=[\mathcal{V}_{\un}]$.
\end{proof}

\subsection{Second step: a phase-parameter transfer phenomenon}

For the sake of simplicity, we let in the sequel $\mathfrak{V}_{k}:=\mathfrak{V}_{\un_k}^{\mathfrak{c}}$, where $\un_k$ is given by the large scale condition. Up to extracting a subsequence, we may assume that $f_{\lambda_0}^{n_k}(c_{j}(\lambda_0))\to x_{j}\in\J_{\lambda_0}$ for all $1\leq j\leq m$. We also let $x:=(x_1,\ldots,x_m)\in(\P^1)^m$.

We want to reduce the problem to a purely dynamical datum of $f$. 
Building on the large scale condition, one actually gets the following.

\begin{proposition}\label{prop:formula3}
There exist $k_0\geq1$ and $\delta,\alpha>0$ such that for any $k\geq k_0$, we have
\[\int_{\Omega_k\times(\P^1)^m}\left(\bigwedge_{j=1}^m(\pi_j)^*\left(\widehat{T}\right)\right)\wedge\left[\mathcal{V}_{\un_k}\right]>0~.\]
\end{proposition}

\begin{proof}[Proof of Proposition~\ref{prop:formula3}]
Set
\[I_k:=\int_{\Omega_k\times(\P^1)^m}\left(\bigwedge_{j=1}^m(\pi_j)^*\left(\widehat{T}\right)\right)\wedge\left[\mathcal{V}_{\un_k}\right]~,\]
and let $\delta$ be given by the large scale condition. Let $S_k$ be the connected component of $\mathcal{V}_{\un_k}\cap \Omega_k\times\D_\delta^m(x)$ containing $(0,\mathfrak{V}_{\un_k}(0))$.  Up to replacing $\delta$ with $\delta/2$, for any $k\geq k_1$, the current $[S_k]/\|[S_k]\|$ is of a vertical current of mass $1$ in $\Lambda\times\D_\delta^m(v_\infty)$ and
\[\supp([S_k])=S_k\subset\Omega_{k}\times\D_\delta^m(x)~.\]
As in the case $m=1$, let $S$ be any weak limit of the sequence $[S_k]/\|[S_k]\|$. Then $S$ is a closed positive $(m,m)$-current of mass $1$ in $\B_\epsilon\times\D^m_\delta(x)$ with $\supp(S)=\{\lambda_0\}\times\D_\delta^m(x)$ by the large scale condition. Hence, by extremality of $[\{\lambda_0\}\times\D^m_\delta(x)]$, we have that  $S=M\cdot[\{\lambda_0\}\times\D^m_\delta(x)]$, where $M^{-1}>0$ is the volume of $\D_\delta^m(x)$ for the volume form $\bigwedge_j\omega_j$, where $\omega_j=(p_j)^*\omega_{\textup{FS}}$ and $p_j:(\P^1)^m\to\P^1$ is the projection on the $j$-th coordinate.

~

As a consequence, $[S_k]/\|[S_k]\|$ weakly converges to $S$ as $k\to\infty$. Since $\bigwedge_{j=1}^m(\pi_j)^*\left(\widehat{T}\right)$ is the wedge product of $(1,1)$ current with continuous potentials, we have 
\[ \bigwedge_{j=1}^m(\pi_j)^*\left(\widehat{T}\right)\wedge\frac{[S_k]}{\|[S_k]\|} \to \bigwedge_{j=1}^m(\pi_j)^*\left(\widehat{T}\right)\wedge S \]
so
\[\liminf_{k\to\infty}\left(\|[S_k]\|^{-1}\cdot I_k\right)\geq\liminf_{k\to\infty}\int\bigwedge_{j=1}^m(\pi_j)^*\left(\widehat{T}\right)\wedge\frac{[S_k]}{\|[S_k]\|}\geq \int\bigwedge_{j=1}^m(\pi_j)^*\left(\widehat{T}\right)\wedge S.\]
 By the above, this gives
\begin{align*}
\liminf_{k\to\infty}\left(\|[S_k]\|^{-1}\cdot I_k\right)&\geq M\cdot \int_{\{\lambda_0\}\times\D_\delta^m(x)}\bigwedge_{j=1}^m(\pi_j)^*\left(\widehat{T}\right)\\
&\geq \int\bigwedge_{j=1}^m(\pi_j)^*\left(\widehat{T}\right)\wedge\left[\{\lambda_0\}\times\D_\delta^m(x)\right].
\end{align*}
The proof of the proposition directly follows from the following lemma.
\end{proof}
\begin{lemma}\label{lm:limit}
For any $\delta>0$, and any $x=(x_1,\ldots,x_m)\in\left(\J_f\right)^m$, we have
\begin{eqnarray*}
\int\bigwedge_{j=1}^m(\pi_j)^*\left(\widehat{T}\right)\wedge\left[\{\lambda_0\}\times\D_\delta^m(x)\right]=\prod_{j=1}^m\mu_{\lambda_0}(\D(x_j,\delta))>0.
\end{eqnarray*} 
\end{lemma}
\begin{proof}
Let us set $\omega_j:=\left((\pi_j)^*\hat\omega\right)|_{\lambda=\lambda_0}$. First, we can remark that $g|_{\lambda=\lambda_0}=g_{\lambda_0}$ is the Green function of the rational map $f_{\lambda_0}$. We denote by $p_j:(\P^1)^k\to\P^1$ the canonical projection onto the $j^{th}$ coordinate. A classical slicing argument gives
$$\left((\pi_j)^*dd^cg\right)|_{\lambda=\lambda_0}=dd^c_z(g_{\lambda_0}\circ p_j)=(p_j)^*dd^cg_{\lambda_0}~.$$
In particular, since $\omega_j=(p_j)^*\omega_{\textup{FS}}$, we have
\[\left.\left((\pi_j)^*(\widehat{T})\right)\right|_{\lambda=\lambda_0}=\left.\left((\pi_j)^*(\omega_{\textup{FS}}+dd^cg)\right)\right|_{\lambda=\lambda_0}=(p_j)^*(\mu_{\lambda_0})~,\]
where $\mu_{\lambda_0}$ is the maximal entropy measure of $f_{\lambda_0}$, hence
\begin{eqnarray*}
I:=\int\bigwedge_{j=1}^m(\pi_j)^*\left(\widehat{T}\right)\wedge\left[\{\lambda_0\}\times\D_\delta^m(x)\right] & = & \int_{\{\lambda_0\}\times \D_\delta^m(x)}\bigwedge_{j=1}^m(\pi_j)^*\left(\widehat{T}\right)\\
& = & \int_{\D_\delta^m(x)}\bigwedge_{j=1}^m\left.\left((\pi_j)^*\left(\widehat{T}\right)\right)\right|_{\lambda=\lambda_0}\\
& = & \int_{\D_\delta^m(x)}\bigwedge_{j=1}^m(p_j)^*\left(\mu_{\lambda_0}\right)~.
\end{eqnarray*}
Since $\supp(\mu_{\lambda_0})=\J_{\lambda_0}$ and $x_1,\ldots,x_m\in\J_{\lambda_0}$, Fubini Theorem yields
$$I = \prod_{j=1}^m\left(\int_{\D(x_j,\delta)}\mu_{\lambda_0}\right)=\prod_{j=1}^m\mu_{\lambda_0}(\D(x_j,\delta))>0~,$$
which ends the proof of Lemma \ref{lm:limit}.
\end{proof}
The proof of Theorem \ref{tm:sumplus} directly follows from Lemma~\ref{lm:formula2} and Proposition~\ref{prop:formula3}.

\section{Misiurewicz maps and the generalized large scale condition}\label{sec:Mis}

Fix $d\geq2$ and pick $1\leq k\leq 2d-2$. A rational map $f\in\rat_d$ is $k$-\emph{Misiurewicz} if the following properties hold:
\begin{itemize}
\item $f$ has no parabolic periodic points,
\item $f$ has $k$ critical points in its Julia set, counted with multiplicity,
\item for any $c\in\mathcal{C}(f)\cap\J_f$, we have $\omega(c)\cap \mathcal{C}(f)=\varnothing$.
\end{itemize}
In this section, we want to emphasize that the large scale condition is the good condition for proving that specific parameters lie in the support of the bifurcation measure. Our motivation here is also to provide a simpler and more intrinsic proof of Theorem~1.4 of \cite{Article1}. Recall that for a critical point $c_j$, we denote by $v_j$ the critical value $f(c_j)$.
According to Theorem~\ref{tm:sumplus}, it is sufficient to prove the existence of a (local) submanifold in which $f$ satisfies the (generalized) large scale condition. More precisely, we prove the following.

\begin{theorem}\label{tm:Misbigscale}
Let $f\in\rat_d$ and $1\leq k\leq 2d-2$. Assume that $f$ is $k$-Misiurewicz 
 and that $\mathcal{C}(f)\cap \J_f=\{c_1,\ldots,c_k\}$. Assume that  for all $n\in \N$ and all $i\ne j \leq k$, $f^n(v_j)\neq v_i$. 
 If $f$ is not a flexible Latt\`es map, then there exists a $k$-dimensional local submanifold $\Lambda_f\subset\mathrm{Rat}_d$ such that $f$ satisfies the generalized large scale condition for $(c_1,\ldots,c_k)$ in $\Lambda_f$.
\end{theorem}

This approach allows us to exhibit the key expansion and distortion arguments, without using more elaborate tools that will necessarily be missing in a more general situation, such as linearizing coordinates along repelling orbits.  

Observe also that the condition  "for all $n\in \N$ and all $i\ne j \leq k$, $f^n(v_j)\neq v_i$" is not an issue: any $k$-\emph{Misiurewicz} map that does not satisfy this condition can be approximated by $k$-\emph{Misiurewicz} maps that do satisfy it using Montel theorem. In particular, it is in the support of the bifurcation current $T_1\wedge \dots \wedge T_k$.

\subsection{Hyperbolic sets and holomorphic motions}
We recall some classical definitions and facts. Let $f\in\rat_d$ and $E\subset\P^1$ be a non-empty compact $f$-invariant set, i.e. such that $f(E)\subset E$. We say that $E$ is $f$-\emph{hyperbolic} if one of the following equivalent conditions is satisfied:
\begin{enumerate}
\item there exist $C>0$ and $\alpha>1$ such that $|(f^n)'(z)|\geq C\alpha^n$ for all $z\in E$ and all $n\geq0$,
\item for some appropriate metric on $\P^1$, there exists $K>1$ such that $|f'(z)|\geq K$ for all $z\in E$. Such a constant $K$ is called a \emph{hyperbolicity constant} for $f$.
\end{enumerate}
Recall also that a \emph{holomorphic motion} of a set $X\subset \P^1$ over a complex manifold $\Lambda$ centered at $\lambda_0\in\Lambda$ is a map $h:\Lambda\times X\rightarrow\P^1$ such that:
\begin{itemize}
\item $h_{\lambda_0}:=h(\lambda_0,\cdot):X\to\P^1$ is the identity map,
\item $h_{\lambda}:=h(\lambda,\cdot):X\to\P^1$ is injective for all $\lambda\in\Lambda$,
\item $\lambda\mapsto h(\lambda,x)$ is holomorphic on $\Lambda$ for all $x\in X$.
\end{itemize}

First, notice that the classical $\lambda$-lemma of Ma\~n\'e-Sad-Sullivan~\cite{MSS} says that any holomorphic motion of $X$ extends continuously to a holomorphic motion of the closure of $X$. It is known that a hyperbolic set $E$ admits a natural holomorphic motion (see, e.g., \cite[Property (1.2) page 229]{Shishikura2} or \cite[Theorem C]{Jonsson-rep}).

\begin{theorem}
Let $(f_\lambda)_{\lambda\in\mathbb{B}(0,r)}$ be a holomorphic family of degree $d$ rational maps parametrized by a ball $\mathbb{B}(0,r)\subset\C^m$. Let $E_0\subset\P^1$ be a compact $f_0$-hyperbolic set. Then there exist $0<\rho\leq r$ and a unique holomorphic motion 
\begin{eqnarray*}
h:\mathbb{B}(0,\rho)\times E_0 & \longrightarrow & \P^1\\
(\lambda,z) & \longmapsto & h_\lambda(z)
\end{eqnarray*}
centered at $0$ and such that $f_\lambda\circ h_\lambda(z)=h_\lambda\circ f_0(z)$ for all $z\in E_0$.\label{mvtholo}
\end{theorem}

The proof of Theorem~\ref{mvtholo} relies on the compactness of $E_0$ and on the next lemma,  which is an immediate corollary of the Implicit Function Theorem.
Let $K>1$ be a hyperbolicity constant for $E_0$ for a suitable metric $\alpha$. In what follows, we denote by $|.|_\alpha$, $\D_\alpha$ the distance and disk with respect to that metric.  Up to reducing $K$ and $r$, we can find a $\delta$-neighborhood $\mathcal{N}_\delta$ of $E_0$ in $\P^1$ such that
\begin{equation}
|f_\lambda'(z)|_\alpha\geq K>1\text{ for all }(z,\lambda)\in\mathcal{N}_\delta\times\mathbb{B}(0,r)~.\label{eq:Misiurewicz}
\end{equation}

\begin{lemma}\label{lm:inversequantitative}
Under the assumption of Theorem~\ref{mvtholo}, there exist $\varepsilon>0$ and $0<\rho\leq r$, such that for all $z_0\in E_0$, there exists a map $f_{z_0,\lambda}^{-1}(w)$ which depends holomorphically on $(\lambda,w)\in \mathbb{B}(0,\rho)\times\D_\alpha(f_0(z_0),\varepsilon)$ and taking values in $\D_\alpha(z_0,\varepsilon)$ which satisfies
\begin{enumerate}
  \item $f^{-1}_{z_0,0}\big(f_0(z_0)\big)=z_0$,
  \item $f_\lambda\big(f^{-1}_{z_0,\lambda}(w)\big)=w$ for all $(\lambda,w)\in\mathbb{B}(0,\rho)\times\D_\alpha(f_0(z_0),\varepsilon)$, and
  \item $\left|\big(f^{-1}_{z_0,\lambda}\big)'(w)\right|_\alpha\leq\frac{1}{K}$ for all $(\lambda,w)\in\mathbb{B}(0,\rho)\times\D_\alpha(f_0(z_0),\varepsilon)$.
\end{enumerate}
\end{lemma}

\subsection{Transversality for Misiurewicz maps}

It is well known that if $f$ is a $k$-Misiurewicz degree $d$ rational map but not a flexible Latt\`es map, then all periodic Fatou components of $f$ are attracting basins and $f$ does not carry any invariant line field on its Julia set. Moreover, there exists a positive integer $n_0\geq1$ such that the set
\[E_f:=\overline{\{f^n(c)\, ; n\geq n_0 \ \text{ and } \ c\in\mathcal{C}(f)\cap\J_f\}}\]
is a compact hyperbolic $f$-invariant set (see e.g.~\cite{Aspenberg1,Article1}).

 If $f\in\rat_d$ is $k$-Misiurewicz,  $\mathcal{C}(f)\cap \J_f=\{c_1,\ldots,c_k\}$ and, for all $n\in \N$ and all $i\ne j \leq k$, $f^n(v_j)\neq v_i$, where $v_i=f(c_i)$, we let $h:\B(f,r)\times E_f\to\P^1$ be the dynamical holomorphic motion of $E_f$. 
Since all the critical points in the Julia set of $f$ are simple ($v_i\neq v_j$), we may follow them holomorphically in the neighborhood of $f$, as well as the critical values: denote by $v_i(\lambda)$, $1 \leq i \leq k$, the corresponding critical values of $f_\lambda$.

Given a holomorphic curve $\lambda \mapsto f_\lambda \in \rat_d$ passing through $f=f_{\lambda_0}$, it will be convenient to denote by $\eta$ the meromorphic vector field on $\mathbb{P}^1$ defined by
$$\eta(z):=\frac{\frac{d}{d\lambda}_{|\lambda=\lambda_0} f_\lambda(z)}{f'(z)}.$$

Recall that the pullback of $\eta$ by $f^n$ is given in coordinates by $(f^n)^*\eta(z) = \frac{\eta \circ f^n(z)}{(f^n)'(z)}$. As before, let $\xi_n^i(\lambda):=f_\lambda^n(v_i(\lambda))$.

\begin{lemma}\label{lem:calculderivxi}
	Let $(f_\lambda)_{\lambda \in \Lambda}$ be a holomorphic curve of rational maps with a marked critical value
	$v_i(\lambda)$.
	Assume that the orbit of $v_i:=v_i(\lambda_0)$ 
	does not meet the critical set of $f:=f_{\lambda_0}$.
	We have, for all $n \geq 1$: : 
	$$\frac{1}{(f^n)'(v_i)}\frac{d}{d\lambda}_{|\lambda=\lambda_0} \xi_n^i(\lambda) = 
	\frac{d}{d\lambda}_{|\lambda=\lambda_0} v_i(\lambda) + \sum_{k=0}^{n-1} (f^k)^*\eta(v_i).$$
	Moreover, if $v_i$ lies in a hyperbolic set $E_f$ and $h_\lambda(v_i)$ denote its holomorphic motion, then
	$$\frac{d}{d\lambda}_{|\lambda=\lambda_0} h_\lambda(v_i) = - \sum_{n=0}^{+\infty} (f^n)^*\eta(v_i).$$
\end{lemma}

\begin{proof}
	To lighten the notations, we will denote by
	$\dot v_i$, $\dot f(z)$, etc. the derivatives $\frac{d}{d\lambda}_{|\lambda=\lambda_0} v_i(\lambda)$,
	$\frac{d}{d\lambda}_{|\lambda=\lambda_0} f_\lambda(z)$, etc.

	We have: $\xi_n^i(\lambda)=f_\lambda^n(v_i(\lambda)) = f_\lambda \circ f_\lambda^{n-1}(v_i(\lambda))$,
	so 
	$$\frac{d}{d\lambda}_{|\lambda=\lambda_0} \xi_n^i(\lambda) = \dot f \circ f^{n-1}(v_i) + f' \circ f^{n-1}(v_i) \cdot \frac{d}{d\lambda}_{|\lambda=\lambda_0} \xi_{n-1}^i(\lambda).$$
	Therefore
	$$\frac{\frac{d}{d\lambda}_{|\lambda=\lambda_0} \xi_n^i(\lambda)}{(f^n)'(v_i)} = \frac{\eta \circ f^{n-1}(v_i)}{(f^{n-1})'(v_i)} + \frac{\frac{d}{d\lambda}_{|\lambda=\lambda_0} \xi_{n-1}^i(\lambda)}{(f^{n-1})'(v_i)} = (f^{n-1})^*\eta(v_i) + \frac{\frac{d}{d\lambda}_{|\lambda=\lambda_0} \xi_{n-1}^i(\lambda)}{(f^{n-1})'(v_i)}.$$
	The first statement then follows by induction on $n$.
	
	Let us now prove the second statement. By Theorem \ref{mvtholo}, there exists a holomorphic family of 
	homeomorphisms $h_\lambda: E_f \to \mathbb{P}^1$, with $h_{\lambda_0}= \mathrm{Id}$ and such that 
	$$h_\lambda \circ f = f_\lambda \circ h_\lambda.$$
	Differentiating with respect to $\lambda$ at $\lambda=\lambda_0$, we get:
	$$\dot h \circ f = \dot f + f' \cdot \dot h,$$
	which we may rewrite as $f^*\dot h - \dot h = \eta$. 
	Since $E_f$ is hyperbolic, the operator $f^*$ is strictly contracting on the space of vector fields on $E_f$,
	and therefore
\[\dot h(v_i) = - \sum_{n=0}^{+\infty} (f^n)^*\eta(v_i),\]
ending the proof.
	\end{proof}

We can deduce the following from Theorem~B of~\cite{Astorg} (see the discussion after Theorem B in~\cite{Astorg}):

\begin{proposition}\label{prop:transM}
Let $f\in\rat_d$ be $k$-Misiurewicz,  $\mathcal{C}(f)\cap \J_f=\{c_1,\ldots,c_k\}$ and, for all $n\in \N$ and all $i\ne j \leq k$, $f^n(v_j)\neq v_i$. Then, there exists a local complex submanifold $\Lambda_f\subset\rat_d$ which contains $f$,
 such that the holomorphic map
	\[\lambda\in\Lambda_f\longmapsto \left(v_1(\lambda)-h_\lambda(v_1),\ldots,v_k(\lambda)-h_\lambda(v_k)\right)\in\C^k\]
	defines local coordinates at $f_0=f$ in $\Lambda_f$. Furthermore, the limits \[\tau_i:=\lim_{n\to\infty} \frac{D\xi_n^i(0)}{(f^n)'(v_i)}\]
	 exist as linear maps on $T_0 \Lambda_f$, for $1\leq i\leq k$, and $\tau_1,\ldots,\tau_k$ are linearly independent.
\end{proposition}

\subsection{The generalized large scale condition for $k$-Misiurewicz maps}

We now turn to the proof of Theorem~\ref{tm:Misbigscale}.

\begin{proof}[Proof of Theorem~\ref{tm:Misbigscale}]
 Consider the set 
	\[\Omega_n:= \left\{ \lambda\in\Lambda_f, \ \forall i\leq k,\   |h_\lambda(v_i) - v_i(\lambda)| \leq \frac{1}{|(f^n)'(v_i)|}\right\}, \]
	then, by Proposition~\ref{prop:transM}, it can written as $L_n^{-1}(\Omega)$ where \[L_n(t_1,\dots, t_k)= \left(\frac{t_1}{|(f^n)'(v_1)|}, \dots,\frac{t_k}{|(f^n)'(v_k)|} \right)  \]
	and $\Omega$ is some polydisk around $f$ in $\Lambda_f$. 
	For all $1 \leq i \leq k$, let $\phi_n^i: \Omega \to \C$ be defined by:
\[  \phi_n^i(\lambda) := \xi_n^i\left( \frac{\lambda}{|(f^n)'(v_i)|} \right) - \xi_n^i(0). \]
Let $\phi_n: \Omega \to \C^k$ be defined by $\phi_n:=(\phi_n^i)_{1 \leq i \leq k}$.

 We start by proving that the sequence $(\phi_n)_{n \in \N}$ is bounded, hence normal, on 
some small polydisk $ \tilde{\rho}\cdot\Omega$.
Let $i \leq k$, let us bound 
$|\xi_n^i(\lambda) - \xi_n^i(0)|$.
We have:
\begin{align}
\nonumber	|\xi_n^i(\lambda) - \xi_n^i(0)| &=  | f_\lambda^n(v_i(\lambda)) - f^n(v_i)|\\
									&\label{eq:controlnormal}\leq |f_\lambda^n(v_i(\lambda)) - f_\lambda^n(h_\lambda(v_i))|
									+ |f_\lambda^n(h_\lambda(v_i)) - f^n(v)|.
\end{align}

By chaining Lemma \ref{lm:inversequantitative},
there exists $\epsilon>0$ such that for all $\lambda$ small enough, for all  $n \in \N$,
there is an inverse branch $g_{n,\lambda}$ of $f_\lambda^n$ that is well-defined on 
$\D_\alpha(f_\lambda^n(h_\lambda(v_i)), \eps )$, and that maps $f_\lambda^n(h_t(v_i))$ to 
$h_\lambda(v_i)$. As the standard metric and the metric $\alpha$ are equivalent, we have that the map $g_{n,\lambda}$ is univalent on some disk $\D(f_\lambda^n(h_\lambda(v_i)), C \cdot\eps)$ for some constant $C$ that does not depend on $n$ nor $\lambda$. 
By Koebe's distortion theorem, the distortion of $g_{n,\lambda}$ is uniformly bounded on $\D(f_\lambda^n(h_\lambda(v_i)), C\cdot\eps/2)$.

We prove by induction on $n$ that there is a constant $0<\tilde{\rho}<1$ depending only on $f$ such that if $|h_\lambda(v_i) - v_i(\lambda)| \leq \frac{\tilde{\rho}}{|(f^n)'(v_i)|}$, then 
$|f_\lambda^n(v_i(\lambda)) - f_\lambda^n(h_\lambda(v_i))|\leq \frac{ C \varepsilon}{4}$. Assume it is true for $n-1$. Observe first that by Lemma 6.6 of \cite{Article1}, there exists $C_1>0$ depending only on $f$ and on the metric, such that we have:
\begin{equation}\label{6.6}
\max \left(  \frac{ |(f_\lambda^{n-1})'(h_\lambda(v_i))|}{|(f^{n-1})'(v_i)|}, \frac{|(f^{n-1})'(v_i)|}{ |(f_\lambda^{n-1})'(h_\lambda(v_i))|} \right) \leq e^{(n-1) C_1 \|\lambda\|} \leq 2,
\end{equation}
by our choice of $\lambda$. 
We can thus apply the induction hypothesis and the bound on the distortion of $g_{n-1,\lambda}$, and then \eqref{6.6} to find:
\begin{align*}
|f_\lambda^{n-1}(v_i(\lambda)) - f_\lambda^{n-1}(h_\lambda(v_i))| &\leq 2 |(f_\lambda^{n-1})'(h_\lambda(v_i))| \cdot 
|v_i(\lambda)-h_\lambda(v_i)|\\
& \leq 4 |(f^{n-1})'(v_i)| \cdot 
|v_i(\lambda)-h_\lambda(v_i)|.
\end{align*}
Let $M:=\sup_{\lambda \in \B(0,\rho)} \|f_\lambda'\|_{\infty}$ and $m:=  \|(f')^{-1}\|_{E_f,\infty}<+\infty$   we deduce:
\begin{align*}
|f_\lambda^{n}(v_i(\lambda)) - f_\lambda^{n}(h_\lambda(v_i))| &\leq M |f_\lambda^{n-1}(v_i(\lambda)) - f_\lambda^{n-1}(h_\lambda(v_i))|\\
                                                              &\leq 4M  |(f^{n-1})'(v_i)| \cdot 
                                                              |v_i(\lambda)-h_\lambda(v_i)|\\
&\leq 4 Mm |(f^{n})'(v_i)| \cdot |v_i(\lambda)-h_\lambda(v_i)| \\
                     &\leq  4 M m |(f^{n})'(v_i)| \cdot \frac{\tilde{\rho}}{|(f^n)'(v_i)|}                      \\
                     &\leq  \frac{ C\varepsilon}{4}
\end{align*}
for $\tilde{\rho}$ small enough, independent of $n$. In particular,
 \begin{equation*}
 |f_\lambda^n(v_i(\lambda)) - f_\lambda^n(h_\lambda(v_i))| = O(1).
 \end{equation*}
It remains to bound the second term of \eqref{eq:controlnormal}.
To do that, just note that 
$|f_\lambda^n(h_\lambda(v_i)) - f^n(v_i)| = |h_\lambda(f^n(v_i)) - f^n(v_i)|$, and that $h_\lambda$
extends continuously to a holomorphic motion of $\mathbb{P}^1$ by the $\lambda$-lemma, hence in particular
$h$ is uniformly continuous on $\B(f,r) \times E_f$.

Therefore $(\phi_n)$ is bounded, hence normal on $\tilde{\rho}\cdot\Omega$. If $\phi$ is a limit of a subsequence of $(\phi_n)$, then Proposition~\ref{prop:transM} implies that $D\phi(0)$ is invertible, hence that $\phi$ is open near $0$. 
So there is some polydisk $\D(0,2\delta)^k$ such that $\D(0,2\delta)^k \subset \phi( \tilde{\rho}\cdot\Omega )$
and therefore, for arbitrarily large $n$, $\D(0,\delta)^k \subset \phi_n(\tilde{\rho}\cdot \Omega)$.
Again, as  $D\phi(0)$ is invertible, its graph is vertical near the origin (restricting $\delta$ if necessary); by normality, this is also the case for $\phi_n$ for $n$ large enough which is exactly what we want.  
\end{proof}

\section{The large scale condition and good Collet-Eckmann maps}\label{sec:distortion}

In the present section, we prove Theorem~\ref{tm:bigscale}
and Theorem \ref{tm:abundance}. We start by recalling the conditions appearing in
Theorem \ref{tm:abundance}, which correspond to "good" Collet-Eckmann maps for which
we will be able to prove that they also satisfy the large scale condition.

	Let $f$ be a rational map, and denote by $\mathcal{C}(f)$ its critical set. 
Recall the Collet-Eckmann condition:

\begin{description}
	\item[Collet-Eckmann (CE)]
	there exist $\ga,\ga_0 > 0$ such that for all $v \in f(\mathcal{C}(f))$ and $n \geq 0$
	$$|(f^n)'(v)| \geq e^{n\ga - \ga_0}$$
\end{description}

and the so-called \emph{backward Collet-Eckmann condition}:
\begin{description}
	\item[Backward Collet-Eckmann (CE2)]
	for all $n \geq 0$ and $x \in f^{-n}(\mathcal{C}(f))$
	$$|(f^{n})'(x)| > e^{n\mu - \mu_0}.$$
\end{description}

Additionally, recall the following conditions appearing
in the papers \cite{Tsu} and \cite{Asp13}:

\begin{description}		
	\item[Basic Assumption (BA)] 
	there exists $\al > 0$ such that for all $v \in f(\mathcal{C}(f))$ and $n \geq 0$
	$$ \ln |f'(f^n(v))| > -n \al.$$
	
	\item[Free Assumption (FA)] 
	there exist $\eta, \iota > 0$ such that for all $v \in f(\mathcal{C}(f))$ and $n > 0$
	$$\sum_{j=0 \atop \dist (f^j(v), \mathcal{C}(f)) \leq \eta}^{n-1} \ln |f'(f^j(v))| > -n \iota.$$

	\item[Aspenberg's Basic Assumption (BA')] 
	there exists $\al > 0$ such that for all $v \in f(\mathcal{C}(f))$, $c \in \mathcal{C}(f)$ and $n \geq 0$
	$$ |f^n(v) - c| > e^{{-n \al}}$$
\end{description}

We will now describe a somewhat more technical key condition from \cite{Asp13} 
(definition 1.15), which we will soon
prove implies the condition (FA) with appropriate constants. Let $\gamma, \gamma_0>0$ and assume 
that $f$ satisfies CE with constants $\gamma,\gamma_0$.
Fix $0 < \be \ll \ga$ and $0 < \de \ll 1$. For some $v \in f(\mathcal{C}(f))$, we say that $f$ makes a \emph{deep return} at time $\nu$ if
$$\dist(f^{\nu}(v), \mathcal{C}(f)) < \de^{2}.$$
The corresponding \emph{bound period} $p$ is defined by
$$p := \min\{k \geq 0 \ :\ |f^{k+1}(f^{\nu}(v)) - f^{k+1}(c)| \geq e^{-\be k}\},$$
where $c$ is the closest critical point to $f^{\nu}(v)$. Only deep returns outside bound periods are considered. Thus, for a given critical value $v$ we can enumerate them
$$\nu_{0}(v) < \nu_{0}(v) + p_{0}(v) < \nu_{1}(v) < \ldots$$
The length of the \emph{free orbit} between two consecutive deep returns is in this case 
$$\mu_{i}(v) := \nu_{i}(v) - (\nu_{i-1}(v) + p_{i-1}(v)) \text{ and } \mu_{0}(v) := \nu_{0}(v).$$

We omit the definition of \emph{shallow returns} in \cite{Asp13}, as they are not needed here.

\begin{description}
	\item [Aspenberg's Free Assumption (FA')] 
	there exist $\de, \be > 0$ and $\tau \in (0,1)$ such that for all $v \in f(\mathcal{C}(f))$ and $s \geq 0$
	$$\sum_{i = 0}^{s} \mu_{i}(v) > (1 - \tau)(\nu_{s}(v) + p_{s}(v)) .$$
\end{description}

We will use the notation BA($\alpha$), CE($\gamma, \gamma_0$), FA($\eta, \iota$), etc. to refer to those conditions with the specified constants $\alpha, \gamma, \eta, \iota,$ etc.

\subsection{Getting a positive measure set of good parameters}\label{sec51}
This subsection is devoted to the proof of Theorem \ref{tm:abundance}.
In order to find a set of positive measure of good Collet-Eckmann rational maps,
our strategy is as follows: first, we deduce from the main result of \cite{Asp13} the existence
of a positive measure set of Collet-Eckmann rational maps satisfying all of the 
previous conditions with appropriate constants.
Then, we follow  Tsujii's presentation and improvement \cite{Tsu} of the Benedicks-Carleson theory for interval maps \cite{BenCar}. This will enable us to obtain precise distortion estimates for those 
good parameters, from which we will deduce that the large scale condition is satisfied.

From now on and through all of this section, $f_0$ is a fixed Collet-Eckmann rational map such
that each critical point of 
$f_0$ is simple. 
We can parametrize analytically a neighborhood $O$ of $f_0$ in parameter space
by means of an analytic map $F: O \times \P^1 \to \P^1$, and up to restricting $O$ we may assume that
each of the following conditions hold for a certain constant $\kappa>0$:

\begin{description}
	\item[(K1)] $\|F\|_{\calC^2} < \ka$.
	\item[(K2)] $\frac 1\ka < \frac{|f_\la'(z)|}{\dist(z, C(f_\la))} < \ka$, if $z \notin C(f_\la)$.
	\item[(K3)] $|\partial_z \log f_\lambda'(z)| < \frac{\kappa}{\dist(z, \mathcal{C}(f_\lambda))}$, if $z \notin \mathcal{C}(f_\lambda)$.
	\item[(K4)] $\dist(c_i(\la), c_j(\la')) > \frac 5\ka$, for all $\la' \in O$ and $i \neq j$.
	\item[(K5)] $\|c_i\|_{\calC^1} < \ka$, for all $i$.
	\item[(K6)] if $0 < r < \kappa^{-1}$ and $c \in \mathcal{C}(f_{\lambda})$, then 
	\[\D(c, \kappa^{-1} \sqrt{r}) \se f_{\lambda}^{-1}(\D(f_{\lambda}(c), r)) \se \D(c, \kappa \sqrt{r}).\] 
\end{description}

We start by recalling the definition of the distortion of an analytic map.
\begin{definition}
	\label{defDist}
	Let $\psi:\Om \to \C$ be holomorphic. We define its \emph{complex distortion} by
	$$\Dist(\psi,\Om) := \sup_{x,y \in \Om} \left| \log \psi'(x) - \log \psi'(y) \right| \in [0,\infty].$$
\end{definition}
We are only interested in small values of the distortion, and therefore $\log$ (referring 
to the principal branch of the logarithm) will always be well-defined on $\psi'(\Om)$.
Let $g_i:U_i \to U_{i+1}$ be holomorphic maps for $i=1,\ldots,n$. Observe that
\begin{equation}
\label{equSDist}
\Dist(g_n \circ \ldots \circ g_1, U_1) \leq \sum_{i=1}^n \Dist(g_i, U_i).
\end{equation}

Assume that $\psi : \D(a, r) \to \C$ has $\Dist(\psi, \D(a, r)) \leq 50^{-1}$. It is relatively easy to check that
\begin{equation}
\label{equImg}
\D\left(\psi(a), \frac {9r}{10} |\psi'(a)| \right) \se \psi(\D(a,r)) \se \D\left(\psi(a), \frac {11r}{10} |\psi'(a)| \right).
\end{equation}

\begin{definition}
	\label{defIDiam}
	Let $U \se \C$ be a connected open set. We define its \emph{inner diameter} as follows
	$$\de(U) := \sup_{x,y \in U} \inf_{x,y \in \Gamma \se U} \len(\Ga),$$
	where the $\inf$ is taken over all curves $\Ga \se U$ which connect $x$ to $y$.
\end{definition}

Combining with the definition of the distortion, we obtain the following upper bound
\begin{equation}
\label{equDDist}
\Dist(\psi, U) \leq \de(U) \sup_{z \in U} |\pa_z \log \psi'(z)|.
\end{equation}

Another natural property of the inner diameter is 

\begin{equation}
\label{equDID}
\de(\psi(U)) \leq \de(U) \sup_{z \in U} |\psi'(z)|.
\end{equation}

For any $\lambda \in O$, $n \in \mathbb{N}^*$ and $z \in \P^1 \sm \bigcup_{j=0}^{n-1} f_\lambda^{-j}(\mathcal{C}(f_\lambda))$, we define
\begin{equation}
\label{equAplus}
a^+=a^+(z,n,f_\la) := \left( 400 e \ka^2 \sum_{j=0}^{n-1} \frac{|(f_\la^j)'(z)|}{|f_\la'(f_\la^j(z) )|} \right)^{-1}.
\end{equation}
As an immediate consequence of the previous definition, we observe
\begin{equation} \label{equApd}
a^+(z,n; f_\la) \cdot |(f_{\la}^{n})'(z)| = \left( 400 e \ka^2 \sum_{j=0}^{n-1} \frac{1}{|(f_\la^{n-j})'(f_\la^j(z) )| \cdot |f_\la'(f_\la^j(z) )|} \right)^{-1}<\frac{1}{400e\ka}.
\end{equation}

\begin{lemma}
	\label{lemDist}
	Let $f$ be a rational map satisfying (K2, K3, K6).
	Let $\D(y,r) \se \D(z, a^+)$. Then for all $0 \leq j < m \leq n$
	$$\Dist\left( f_\la^{m-j}, f_\la^j(\D(y,r)) \right) < \frac r{100 a^+}.$$
\end{lemma}

\begin{proof}
	We show by induction that for all $j=0,\ldots,n-1$
	\begin{equation}
	\label{equTemp1}
	\Dist\left(f_\la, f_\la^j(\D(y,r))\right) < 4 e \ka^2 r\frac{|(f_\la^j)'(z)|}{|f_\la'(f_\la^j(z))|},
	\end{equation}
	and conclude by the definition \eqref{equAplus} of $a^+$ and the bound \eqref{equSDist}. Let $0 \leq j < n$ and assume \eqref{equTemp1} holds for all $0 \leq i < j$ and all $\D(y,r) \se \D(z,a^+)$. Using \eqref{equSDist},
	$$\Dist(f_\la^j, \D(z,a^+)) \leq \sum_{i=0}^{j-1} \Dist\left(f_\la, f_\la^i(\D(z,a^+))\right) < 10^{-2},$$
	thus, by (K2)
	$$\frac{\diam(f_\la^j(\D(z,a^+)))}{\dist(f_\la^j(z), C(f_\la))} < 2 e^{0.01} \ka a^+ \frac{|(f_\la^j)'(z)|}{|f_\la'(f_\la^j(z))|} < 10^{-2}.$$
	As a consequence of the above inequalities, using bounds \eqref{equDDist}, (K3), (K2) and \eqref{equDID}, we obtain
\begin{align*}
\Dist(f_\la,f_\la^j(\D(y, r))) & < \ka \frac{\de(f_\la^j(\D(y, r)))}{\dist(f_\la^j(\D(y, r)), C(f_\la))}  
	\leq \ka \frac{\de(f_\la^j(\D(y, r)))}{\dist(f_\la^j(\D(z, a^+)), C(f_\la))}\\
	& <1.01 \ka \frac{\de(f_\la^j(\D(y, r)))}{\dist(f_\la^j(z), C(f_\la))}< 
	1.01 e^{0.01} \ka^2 \frac{|(f_\la^j)'(y)| \de(\D(y,r))}{|f_\la'(f_\la^j(z))|}\\
	& <e^{0.03} \ka^2  \frac{2r |(f_\la^j)'(z)|}{|f_\la'(f_\la^j(z))|},
\end{align*}
	which completes the inductive proof.
\end{proof}
%
%$$\Dist(f_\la,f_\la^j(\D(y, r))) < \ka \frac{\de(f_\la^j(\D(y, r)))}{\dist(f_\la^j(\D(y, r)), C(f_\la))}  
%	\leq \ka \frac{\de(f_\la^j(\D(y, r)))}{\dist(f_\la^j(\D(z, a^+)), C(f_\la))} < $$ 
%	$$1.01 \ka \frac{\de(f_\la^j(\D(y, r)))}{\dist(f_\la^j(z), C(f_\la))}< 
%	1.01 e^{0.01} \ka^2 \frac{|(f_\la^j)'(y)| \de(\D(y,r))}{|f_\la'(f_\la^j(z))|} < e^{0.03} \ka^2  \frac{2r |(f_\la^j)'(z)|}{|f_\la'(f_\la^j(z))|},$$

For any set $A \se \P^1$ and any $r>0$, we denote the $r$-neighborhood of $A$ by
$$B(A,r) := \bigcup_{z \in A} \D(z,r)$$
and its punctured neighborhood by
$$B(A,r)^* := B(A,r) \sm A.$$ 
Let $C(f, \de) := B(\mathcal{C}(f), \de)$ and $C(f,\de)^* := B(\mathcal{C}(f),\de)^*$.

For $z \in C(f, \ka^{-1})^*$, we define
\begin{equation}
\label{equKx}
k(z) := \min\{k \geq 1\, ; \ \log|(f^k)'(f(c))| > 1 - 1.9 \log|f'(z)|\},
\end{equation}
where $c$ is the closest critical point of $f$ to $z$.

\begin{lemma}
	\label{lemKx}
	Let $\alpha, \gamma, \gamma_0, \mu,\mu_0>0$, and let $f$ be a rational map satisfying BA($\alpha$),
	CE($\gamma,\gamma_0$) and CE2($\mu,\mu_0$).
	If $\al < \frac \ga{200}$, there exist $\de_0 > 0$ such that for all $0 < \de < \de_0$ and $z \in C(f, \de)^*$
	\begin{enumerate}
		\item[$(1)$] $- \frac{\log |f'(z)|}{\log \ka} < k(z) + 1 < -2 \frac{\log |f'(z)|}{\ga}$,
		%		\item \footnote{On s'en sert vraiment? Sinon, on peut peut-être le supprimer?} $j\ga - \ga_0 - 1 < \log|(f^j)'(f(z))| < -2 \log|f'(z)|$ for all $0 \leq j \leq k(z)$,
		\item[$(2)$] $\log |(f^{k(z)+1})'(z)| > -0.9 \log|f'(z)| > \frac \ga 3(k(z)+1) + \ga_0 + \mu_0 + 1$,
		\item[$(3)$] $\log|f'(f^j(z))| > 0.1 \log|f'(z)|$ for all $0 < j \leq k(z)$,
		\item[$(4)$] $\left| \log \left| \frac{f^{j}(f(z)) - f^{j}(f(c))}{f(z) - f(c)} \right| - \log |(f^{j})'(f(c))|\right| \leq 1$, for all $0 \leq j \leq k(z)$.
	\end{enumerate}
\end{lemma}

\begin{proof} Let $c$ be the critical point which is closest to $z$ and $v := f(c)$. Claim (1) follows from (K1) and (CE). From (BA) we deduce that
	$$ \log|f'(f^i(v))| > -k(z) \frac \ga{200}, \text{ for all } 0 \leq i \leq k(z),$$
	provided $\de$ is sufficiently small and thus $k(z)$ is large.
	Combined with claim (1), we obtain
	$$\log|f'(f^i(v))| > 10^{-2} \log|f'(z)|, \text{ for all } 0 \leq i \leq k(z).$$
	By (K2)
	\begin{equation}
	\label{equK3}
	\dist(f(z), v) < \ka^3 |f'(z)|^2.
	\end{equation}
	
	By the definition \eqref{equAplus}, we immediately obtain
	\begin{equation} \label{equAMax}
	a^+(z,n; f) \geq \left( 400 e \ka^2 n \max_{0 \leq j < n} \frac{|(f^j)'(z)|}{|f'(f^j(z) )|} \right)^{-1}.
	\end{equation}
	Combining the previous three inequalities with the definition \eqref{equKx} of $k(z)$, we obtain
	\begin{equation}
	\label{equAP}
	\dist(f(z), v) < a^+(v, k(z); f).
	\end{equation}
	All other claims follow from Lemma~\ref{lemDist} in $v$.
\end{proof}

\begin{remark}\normalfont
	An inspection of the proof of Lemma \ref{lemKx} shows that we did not in fact use the 
	property CE2. However, in further applications of this lemma, we will apply it to rational 
	maps that do satisfy CE2, and we will need the estimate from item (2) in which $\mu_0$
	appears.
\end{remark}

We are now able to reconcile the difference in the definition of the conditions 
FA and FA': 

\begin{lemma}
	\label{lemFA}
	Let $\gamma,\gamma_0, \mu,\mu_0>0$, let $\alpha<\gamma/200$ and let $f_\lambda \in O$ satisfying 
	CE($\gamma,\gamma_0$), CE2($\mu,\mu_0$), BA'($\alpha$) and (K1-6). Let $\tau>0$.
	There exist $\hat{\delta}>0$ and $\hat \beta>0$ such that if $f_\lambda$ satisfies FA'($\delta, \beta, \tau$) 
	with $\delta<\hat \delta$ and $\beta<\hat{\beta}$,
	then $f_\lambda$ satisfies FA($\eta, \iota$) with $\eta:=\delta^2$ and 
	$\iota:=\tau \log \kappa \cdot \max \left( 1,  3 \frac{\beta+\log \kappa}{\gamma+\beta})-1  \right)$.
\end{lemma}

\begin{proof}
	Let $\hat{\delta}^2:=\delta_0$ from Lemma~\ref{lemKx}, and let $\hat \beta:=0.05 \gamma$.
	Let $f:=f_\lambda$ be a rational map satisfying all the above conditions and 
	FA'($\beta,\delta, \tau$) with $\delta<\hat \delta$ and $\beta < \hat \beta$.
	Let $\nu$ be a deep return for a critical value $v'$, i.e.
	$\dist(f^\nu(v'),c) \leq \delta^2$ for some critical point $c$. Let $v=f(c)$.
	
	%	 By the definition \eqref{equKx} of $k(z)$ and inclusions \requ{Img} and \eqref{equAP}, observe that $k(f^\nu(v)) + 1 < p$, the length of the bound period.
	
	Let $k_0:=k(f^\nu(v'))$ and $z_0=f^{\nu+1}(v')$. Then, by \eqref{equAP},  $z_0 \in D:= \mathbb{D}(v,a^+(v, k(f^\nu(v'))))$. We claim that 
	\begin{equation}\label{eq:k<p}
	k_0 + 1 < p	
	\end{equation}
	where $p$ is the length of the bound period.
	In order to prove the claim, we need to prove that for all $j \leq k_0$, $|f^{j+\nu}(v')-f^j(c)|<e^{-\beta j}$ (by definition of $p$).
	Let $z:=f^\nu(v)$.
	
	By  Lemma~\ref{lemKx}(4), we have that for all $j \leq k_0$, 
	\begin{equation*}
	|f^{j+\nu}(v')-f^j(c)| \leq e |(f^j)'(f(c))| \cdot |f(z)-f(c)|
	\end{equation*}
	and by (K6), we have $\dist(f(z),f(c))\leq \kappa \dist(z,c)^2 \leq \kappa^3 |f'(z)|^2$. 
	By definition of $k_0$, we also have $|(f^j)'(f(c))| \leq e \cdot |f'(z)|^{-1.9}$.
	Therefore we have
	\begin{equation*}
	|f^{j+\nu}(v')-f^j(c)| \leq \kappa^3 |f'(z)|^{0.1}.
	\end{equation*}
	On the other hand, by Lemma~\ref{lemKx}(1),
	\begin{equation*}
	e^{-\beta j} \geq e^{-\beta k_0} \geq |f'(z)|^{2\beta/\gamma}.
	\end{equation*}
	By the choice of $\hat \beta$, we have $\frac{2\beta}{\gamma}\leq 0.1$, and therefore, 
	up to taking an even smaller $\hat \delta$, we have proved the claim (\ref{eq:k<p}).
	
	We consider the following sets:
	\begin{eqnarray*}
		V_k&:= D\cap f^{-k}(\mathbb{D}(f^k(v),3e^{-k\beta })), \ &V^k:= \cap_{j=k_0}^k V_j, \ V:= V^p   \\
		U_k&:= D\cap f^{-k}\left(\mathbb{D}\left(v,\frac{2e^{-k\beta}}{|(f^k)'(v)| }\right)\right), \ &U^k:= \cap_{j=k_0}^k U_j, \  U:= U^p \\
		W_k&:= D\cap f^{-k}(\mathbb{D}(f^k(v),e^{-k\beta })), \ &W^k:= \cap_{j=k_0}^k W_j, \ W:= W^p 
	\end{eqnarray*}
	Let us show that for all  $k_0 \leq k \leq p$
	\begin{equation}
	\label{equP2}
	\Dist(f^k, V^k) < \frac 1{50}.
	\end{equation}
	By inequality \eqref{equSDist}:
	\begin{align*}
	\Dist(f^k, V^k) &\leq \Dist(f^{k_0}, V^k) + \sum_{j=k_0}^{k-1} \Dist(f,  f^j(V^k))\\
	&\leq  \Dist(f^{k_0}, D) + \sum_{j=k_0}^{k-1} \Dist(f, \mathbb{D}(f^j(v),3e^{-j\beta })  ).
	\end{align*}
	By Lemma~\ref{lemDist}, the first term is $<10^{-2}$. By Lemma 2.1 in \cite{Asp13}, it is enough to show that for all $y \in \mathbb{D}(f^j(v),3e^{-j\beta })$
	$$\sum_{j=k_0}^{k-1}\left| \frac{f'(y)}{f'(f^{j}(v))} - 1\right| < 10^{-3}.$$
	By (BA) and (K1), for  $y \in \mathbb{D}(f^j(v),3e^{-j\beta })$, one has: 
	\begin{align*}
	\left| \frac{f'(y)}{f'(f^{j}(v))} - 1\right|  <    \kappa 3 e^{-j (\beta -\alpha)}.
	\end{align*}	
	By Lemma~\ref{lemKx}(1) and (K2), $k_0 > -\frac{\log |f'(f^\nu(v))|}{\log \ka} - 1 >  -\frac{2\log \de}{\log \ka} - 2$. So, by Lemma 2.1 in \cite{Asp13}:
	\begin{align*}
	\sum_{j=k_0}^{k-1}\left| \frac{f'(y)}{f'(f^{j}(v))} - 1\right|  &< \sum_{j = k_0}^{k-1} \ka e^{-j(\be - \al)} \\
	&< \frac{3\ka e^{-k_0(\be-\al)}}{1 - e^{-(\be - \al)}} < \frac{\ka e^2 \de^{\frac{2(\be-\al)}{\log \ka}}}{1 - e^{-(\be - \al)}}
	\end{align*}
	which completes our argument, as $\de$ can be taken arbitrarily small. 
	
	We show that $ U^k \subset V^k$ for all $k_0\leq k \leq p$. If not, let $k_0 < k$ such that there is $x \in U^k\backslash V^k$. Taking $k$ and $x$ minimal, we assume that $[x,v]\in V^{k-1}$, $[v,x[ \subset V^k$ and $x \in \partial V_k$. In particular, 
	\[\Dist(f^k, [v,x]) \leq \frac{1}{50}    \]   
	hence, as $x \in U_k$:
	\[3e^{-k\beta}= |f^k(x)-f^k(v)|\leq 1.03 |x-v|. |(f^k)'(v)|\leq 2.06e^{-k\beta}, \]
	a contradiction.
	
	We show  now that $W^k \subset U^k$ for all $k_0\leq k \leq p$. If not, let $k_0 < k$ and $x \in W^k$ minimal, such that $[x,v]\in U^{k}\subset V^k$, $[v,x[ \subset V^k$ and $x \in \partial U_k$. In particular, 
	\[\Dist(f^k, [v,x]) \leq \frac{1}{50}    \]   
	hence, as $x \in \partial U_k$:
	\[e^{-k\beta} >|f^k(x)-f^k(v)|\leq 0.97 |x-v|. |(f^k)'(v)|= 1.94e^{-k\beta}, \]
	a contradiction.
	
	In particular, we have $W\subset U \subset V$ therefore $[v,z_0] \subset V$ which implies that $\forall k_0\leq k \leq p$:
	\begin{equation}\label{equDistonV}
	0.97.|z_0-v|.|(f^k)'(v)|< |f^k(z_0)-f^k(v)|   < 1.03.|z_0-v|.|(f^k)'(z_0)|
	\end{equation}

	We have seen that $p > k_0 + 1 > -\frac{\log|f'(f^\nu(v))|}{\log\ka}$. Thus, by (FA'):
	\begin{equation}
	\label{equBFA}
	\sum_{i=0}^{s} -\log|f'(f^{\nu_i(v)}(v))| < \log \ka \sum_{i=0}^{s} p_i(v) < \tau\log\ka  (\nu_s(v) + p_s(v)).
	\end{equation}
	By inequalities (K2) and \requ{DistonV} applied for $j=p-1$, (CE) and the definition of $p$,
	$$\ka^{-2}|f'(z)|^2 < |f(z)-f(c)| < e^{\ga + \ga_0 - p(\ga+\be) + 1},$$
	thus
	$$p < -\frac 3{\ga + \be} \log|f'(z)|.$$
	Let $\Si_1$ and $\Si_2$ be the sums of $\log|f'(f^j(z))|$ respectively over the sets
	$$\{0 < j \leq p\,:\, f^j(z) \in C(f, \de^2)\} \text{ and } \{0 < j \leq p\,:\, f^j(z) \notin C(f, \de^2)\}.$$
	Observe that $\Si := \Si_1 + \Si_2 = \log|(f^p)'(z_0)|$. By (K2), inequalities  \requ{DistonV} and \requ{Img} and the definition of $p$, we have
	$$\Si + 2\log|f'(z)| > \Si + \log|f(z)-f(c)| - 2\log\ka \geq -p\be -3\log\ka,$$
	thus by (K1)
	$$-\Si_1 < \Si_2 + 2\log|f'(z)| + p\be + 3\log\ka < \log|f'(z)| + p(\be + \log\ka)$$
	$$ < -\log|f'(z)|\left(3\frac{\be + \log\ka}{\ga + \be} - 1\right).$$
	Combined with inequality \requ{BFA} it shows that Aspenberg's (FA') implies (FA) with constants $\eta := \de^2$ and $\iota := \tau \log \ka \cdot \max \left(1, 3\frac{\be + \log\ka}{\ga + \be} - 1\right)$.
\end{proof}

By the main result of \cite{GraSmi}, Collet-Eckmann rational maps satisfy the backward Collet-Eckmann condition at critical points of maximal multiplicity. In our setting, all critical points of $f_0$ are simple. However, we need uniform constants for the backward contraction in a neighborhood of $O$ of 
$f_0$. 

Let $\gamma, \gamma_0>0$ such that $f_0$ satisfies CE($\gamma,\gamma_0$).
A careful inspection of the proof of Proposition 1 in \cite{GraSmi} reveals the following facts. Dynamically defined constants $M, R'$ and $R$ can be chosen to be uniform in a neighborhood of $f_0$. They also depend on $\la_1$ and $C_1$, which in our setting are expressed by constants $\ga, \ga_0$. A $\calC^1$ continuity argument, similar to the proof of the above proposition, shows that constants $C_{2t}$ and $L$ can be chosen uniformly for maps $f_\lambda$ satisfying CE($\gamma,\gamma_0$)
in a neighborhood of $f_0$. Observe that the constant $K_{3.1}$ only depends on $R'$.

This observation proves the following:
\begin{proposition}
	\label{propCE2}
	There exist $O$, a neighborhood of $f_0$ in parameter space, and constants $\mu, \mu_0 > 0$ such that every $f_\la$ with $\la \in O$ satisfying CE($\gamma,\gamma_0$) also satisfies CE2($\mu,\mu_0$).
\end{proposition}

We now restate for the reader's convenience the main result from \cite{Asp13}:

\begin{theorem}[Aspenberg, \cite{Asp13}]\label{th:mainaspenberg}
	Let $f_0$ be a strongly Misiurewicz rational map with simple critical points. If $f_0$ is not a flexible Latt\`es map,
	there exist $\gamma, \gamma_0>0$ and
	$\hat a, \hat \delta, \hat \tau>0$ such that for all $\alpha < \hat \alpha, \delta < \hat \delta, 
	\tau < \hat \tau$, the map $f_0$ is a Lebesgue point of density of rational maps 
	satisfying CE($\gamma,\gamma_0$), BA'($\alpha$), and FA'($\beta, \delta, \tau$) 
	with $\beta:=10\alpha$. 
\end{theorem}

Putting together those results, we are now ready to prove Theorem \ref{tm:abundance}:

\begin{proof}[Proof of Theorem \ref{tm:abundance}]
	Let $\gamma, \gamma_0>0$ be the constants given by Theorem \ref{th:mainaspenberg}.
	Let $\mu, \mu_0>0$ be the constants given by Proposition \ref{propCE2}. Let $\hat \eta:=\hat{\delta}^2$ and 
	$$\hat \iota:= \sup_{0<\beta\leq 10\hat \alpha}\hat \tau \log \kappa \cdot  \max \left( 1,  3 \frac{\beta+\log \kappa}{\gamma+\beta}-1  \right).$$
	
	Let 
	$\alpha<\min(\hat \alpha,\gamma/200)$, $\eta<\hat \eta$ and $\iota< \hat \iota$.
	
	Let $E$ be the set of rational maps satisfying CE($\gamma,\gamma_0$), BA'($\alpha/\kappa$), and FA'($\beta, \delta, \tau$)  with $\beta:=10\alpha$, $\delta^2:=\eta$, and 
	$\tau>0$ small enough to ensure that 
	$\tau \max \left( 1,  3 \frac{\beta+\log \kappa}{\gamma+\beta})-1  \right) \leq \iota$.
	
	By Lemma \ref{lemFA}, each $f \in E$ satisfies FA($\eta, \iota$).
	By Proposition \ref{propCE2}, each $f \in E$ satisfies CE2($\mu,\mu_0$).
	By (K2), each $f \in E$ satisfies BA($\alpha$).
	By Theorem \ref{th:mainaspenberg}, $f_0$ is a Lebesgue density point of $E$.
\end{proof}

\subsection{Proving the large scale condition for good parameters}

In the present section, we prove Theorem~\ref{tm:bigscale}.
Let $\mathcal{F}$ be the space of complex lines passing through $0$ in $\C^N$ and let $\mathcal{L}$ be the unique volume form on $\mathcal{F}$ which is invariant under the action of the unitary group $\mathbb{U}(N)$ and of total mass $1$. We shall rely on the following result of Sibony and Wong~\cite{SW}.

\begin{theorem}[Sibony-Wong]\label{tm:sibonywong}
Let $m>0$ be a positive constant. Let $\mathcal{F}'\subset\mathcal{F}$ be such that $\mathcal{L}(\mathcal{F}')\geq m$ and let $\Sigma$ denote the intersection of the family $\mathcal{F}'$ with the ball $\B(0,r)$ of $\mathbb{C}^N$. Then any holomorphic function $h$ on a neighborhood of $\Sigma$ can be extended to a holomorphic function on $\B(0,\tilde{\rho} r)$ where $0<\tilde{\rho}\leq 1$ is a constant depending on $m$ and $N$ but independent of $\mathcal{F}'$ and $r$. Moreover, we have
\[\sup_{\B(0,\tilde\rho r)}|h|\leq \sup_{\Sigma}|h|.\]
\end{theorem}
Another key ingredient will be a transversality property of
the differentials of the $\xi_n^i$. This property first appeared in the work of Tsujii in the real dynamics 
case. It was proved for some class of rational maps (containing notably Collet-Eckmann polynomials) by Levin in \cite{levin2014perturbations}. In \cite{Astorg}, this property was extended to rational
maps having no invariant line fields, and so (see Lemma \ref{lm:noinvlinfield}) this property is satisfied by all
Collet-Eckmann rational maps.

From now on, we will not rely on the result of Aspenberg (\cite{Asp13}). 
We will also not make any more use of the conditions FA' or BA'.
Instead, we fix in this whole 
subsection some rational map $f$ satisfying CE($\gamma,\gamma_0$),  CE2($\mu,\mu_0$), 
BA($\alpha$) and FA($\eta, \iota$). We denote by $\mathbb{S}_{2d-2}$ the unit sphere in $\mathbb{C}^{2d-2}$.

The main goal of this section is to prove the following distortion estimate 
for rational maps $f$ satisfying all the previous conditions:

\begin{proposition}\label{th:distortionestimate}
	Let $\gamma, \gamma_0, \mu, \mu_0>0$ and $\alpha<\gamma/200$. Let $\eta>0$.
	There exists $\iota>0$ such that if $f$
	satisfy	CE($\gamma,\gamma_0$), CE2($\mu,\mu_0$), BA($\alpha$) and  FA($\eta, \iota$),
	then
	there exist $C>0$, a set $E \subset \mathbb{S}_{2d-2}$ such that $\lambda(E)>0$ and an infinite set $G \se \mathbb{N}$ of natural numbers with the following property: for every $n \in G$ and all $c_i \in \mathcal{C}(f)$ there exists $r_{n,i} > 0$ such that for every $u\in E$ 
	$$\Dist(\xi_n^i, \D(0, r_{n,i}) \cdot u) < 10^{-2}, \text{ and}$$
	$$\frac{1}{C} <r_{n,i} |(f^n)'(f(c_i))| < C.$$
\end{proposition}

\subsection{Distortion estimates for a single CE map}

\begin{lemma}
	\label{lemBkDil}
	Assume $\al < \ga/200$. There exist $\de_1 > 0$ such that for all $0 < \de < \de_1$, if $n>0$ and $z \in \P^1$ satisfy
	\begin{equation}\label{equBK}
	f^j(z) \notin C(f, \de) \text{ for all } 0 \leq j < n \text{ and } f^n(z) \in C(f, \de),
	\end{equation}
	we have
	$$|(f^n)'(z)| > e^{n\mu - \mu_0 - 1}.$$
\end{lemma}
\begin{proof}
	For $\de > 0$ and $n \in \N$ denote
	$$A(n,\de) := f^{-n}(\mathcal{C}(f)) \sm \bigcup_{j = 0}^{n-1} f^{-j}(C(f, \de^{2})).$$
	Let us prove that if $W$ is a connected component of $f^{-n}(C(f,\de))$ which intersects $A(n,\de)$ then 
	\begin{equation} \label{equDW}
	\Dist(f^{n}, W) < 10^{-2}.
	\end{equation} 
	Let $z \in W \cap A(n,\de)$, $l := -2\mu^{-1}\log \de$ and $\Si_{1},\Si_{2},\Si_{3}$ the sums of $$|(f^{n-j})'(f^{j}(z)) \cdot f'(f^{j}(z))|^{-1}$$ on the following sets of $j \in \N$, respectively:
	$$\{n-l \leq j < n\ :\ f^{j}(z) \in C(f, \sqrt \de)\},$$
	$$\{n-l \leq j < n\ :\ f^{j}(z) \notin C(f, \sqrt \de)\},\ \{0 \leq j < n-l\}.$$
	If $\de$ is small and $y \in C(f, \sqrt \de)$ then Lemma~\ref{lemKx}, claims (3) and (4), provides $k > 0$ such that $\log |(f^{k+1})'(y)| > -0.9 \log |f'(y)|$ and for all $0 < j \leq k$, $f^{j}(y) \notin C(f, \de)$. Combined with (K2) and (CE2) at $f^n(z) \in \mathcal{C}(f)$,
	we obtain
	$$\Si_{1} < l \cdot \left( \frac{\de^{2}} \ka \right)^{-0.1} e^{\mu_{0}},\ \Si_{2} < l \cdot \ka \de^{-0.5} e^{\mu_{0}},\ \Si_{3} < \sum_{j > l} \ka\de^{-2} e^{- j \mu  + \mu_{0}}\ .$$
	By inequality \requ{Apd}
	$$a^+(z,n,f) \cdot |(f^{n})'(z)| = (400e\ka^{2}(\Si_{1} + \Si_{2} + \Si_{3}))^{-1} > 2\de\ .$$
	Lemma~\ref{lemDist} and observation \requ{Img} complete the proof of claim \requ{DW}.
	
	Let $z \in \P^1$ satisfying \requ{BK} for some $n > 0$. Let $0 \leq k_{0} \leq n$ be the smallest integer such that $f^{k_{0}}(z) \in W$, a connected component of $f^{-n+k_{0}}(C(f,\de))$ which contains $z' \in A(n-k_{0}, \de)$. We prove that $k_{0} = 0$ and use inequality \requ{DW} and (CE2) to complete the proof.
	
	Suppose $k_{0} > 0$. Let $J$ be the connected component of $f^{-1}(W)$ containing $y := f^{k_{0} - 1}(z)$. There are two possibilities.
	
	\noindent\textbf{Case 1.} There exists $c \in \mathcal{C}(f) \cap J$. Then by \requ{BK}, $y \notin \D(c, \de)$ and by (K6) $\diam(W) \geq \dist(f(y), f(c)) > \ka^{-2} \de^{2}$. Then by \requ{DW} and \requ{Img}
	$$|(f^{n-k_{0}})'(f(c))| < 2\frac{2\de}{\diam(W)} < 4 \ka^{2} \de^{-1}.$$
	Then by (CE)
	$$\ga (n-k_{0}) < \log |(f^{n-k_{0}})'(f(c))| + \ga_{0} <  \log( \de^{-1}) + \log (4 \ka^{2}) + \ga_{0}.$$
	As $f^{n-k_{0} + 1}(c) \in C(f, \de)$, by (BA) and (K2) we obtain
	$$ -(n - k_0)\al < \log|f'(f^{n-k_{0} + 1}(c))| < \log\de + \log\ka.$$
	As $\al < \ga/200$, taking $\de$ sufficiently small, we obtain a contradiction. 
	
	\noindent\textbf{Case 2.} $J \cap \mathcal{C}(f) = \varnothing$. Let $z'' := (f|_{J})^{-1}(z') \notin A(n-k_{0} + 1,\de)$. There exists  $c \in \mathcal{C}(f)$ with $\dist(z'', c) < \de^{2}$. As $c \notin J$, we may choose $x \in [z'', c] \cap \pa J$. By (K6), $\dist(z', f(x)) < \ka^{2}\de^{4}$. As $f(x) \in \pa W$, by \requ{DW} and \requ{Img} applied to $f^{-(n-k_{0})}$
	$$W \se \D(z', 2\, \ka^{2}\de^{4}) \text{ and hence } J \se \D(z'', \sqrt 2\, \ka^{2}\de^{2}).$$
	But this implies $\dist(y, c) \leq \diam(J) + \dist(z'', c) < \de$, which contradicts \requ{BK}.	
\end{proof}

Let $\de_2 := \min(\de_0, \de_1)$ such that lemmas \ref{lemKx} and \ref{lemBkDil} apply for all $\de < \de_2$.

We will use the following variant of a well known result due to Ma\~né \cite{Mane}. We provide a short proof using the $\mathrm{ExpShrink}$ property of (CE) maps \cite[Main Theorem]{PrzRLeSmi}:
there exist $r, \ga'>0$ such that for all $z \in \mathcal J$ (the Julia set of $f$), every $n > 0$ and every connected component $W$ of $f^{-n}(\D(z, r))$ 
$$\diam(W) < e^{-n\ga'}.$$

\begin{proposition}\label{propMane} 
	Let $K:= \P^1 \sm C(f, \de_2)$ be a compact set. There exist $ \ga', \ga_0' > 0$ such that for all $n \in \N$ and $z \in \cap_{i=0}^{n-1} 	f^{-i}(K)$
	$$|(f^n)'(z)| > e^{n\ga' - \ga_0'}.$$
\end{proposition}

\begin{proof}
	As all critical orbits of $f$ satisfy the Collet-Eckmann condition, by the classification of Fatou domains, $\mathcal J = \P^1$. Using $\mathrm{ExpShrink}$, by continuity and by eventually shrinking $r$, we may assume that for all $i < n$
	$$\diam(f^i(W)) < \frac{\dist(K, \mathcal{C}(f))}2,$$
	where $W$ is an arbitrary connected component of $f^{-n}(\D(z, r))$.
	Then $f^n$ is univalent on $W$ and we conclude by the Koebe lemma applied to $f^{-n}$ on $\D(z, r/2)$ that if $n$ is large enough and $f^i(z) \in K$ for $i=0,\ldots,n-1$ then
	$$|(f^n)'(z)| > 3.$$
	which completes the proof.
\end{proof}

Let $\ga'$ and $\ga_0'$ be provided by Proposition~\ref{propMane}. Let
\begin{equation}
\label{equSi}
\si := \min(\ga / 3, \mu, \ga') \text{ and } \si_0 := \log\ka + \ga' + \mu_0 + \ga_0' + 1.
\end{equation}

\begin{lemma}\label{lemDfn}
	Let $z \in \P^1$ and $n > 0$ such that for all $0 \leq i < n$, $f^i(z) \notin \mathcal{C}(f)$. Let $I := \{ 0 \leq i < n\ :\ f^i(z) \in C(f, \de_2) \text{ and } i + k(f^i(z)) \geq n\}$. Let $\be := 1$ if $I = \varnothing$, otherwise $\be := \min\{\d(f^i(z), \mathcal{C}(f)) \ :\ i \in I\}$. Then
	$$\log |(f^n)'(z)| > \log \be + n\si - \si_0.$$
\end{lemma}
\begin{proof}
	If for all $0 \leq i < n$, $f^i(z) \notin C(f, \de_2)$, $I = \varnothing$ and we apply Proposition~\ref{propMane} to complete the proof. 
	
	Let $i_1 := \min\{0 \leq i < n \, :\, f^i(z) \in C(f, \de_2)\}$ and for each $p \geq 1$, $i_{p+1} := \min\{i_p + k(f^{i_p}(z)) < i < n \, :\, f^i(z) \in C(f, \de_2)\}$. Let $s$ be the largest index such that $i_s < \infty$.
	
	We use Lemma~\ref{lemBkDil} at $z$ and obtain
	$$\log |(f^{i_1})'(z)| > i_1\mu - \mu_0 - 1 \geq i_1 \si - \mu_0 - 1.$$
	Let $k_p := k(f^{i_p}(z))$ for all $1 \leq p \leq s$. By Lemma~\ref{lemKx}(2), we get
	$$\log |(f^{k_1 + 1})'(f^{i_1}(z))| > (k_1 + 1)\frac \ga 3 + \mu_0 + 1 \geq (k_1 + 1)\si + \mu_0 + 1.$$
	We iterate this chaining of lemmas \ref{lemBkDil} and \ref{lemKx} from $i_{p-1} + k_{p-1} + 1$ to $i_{p}$ and respectively from $i_{p}$ to $i_{p} + k_{p} + 1$. 
	
	The only exception occurs for $p=s$ if $i_s \in I$. In this case, we stop at $i_s$ and estimate directly by (K2)
	$$\log |f'(f^{i_s}(z))| > \log \d(f^{i_s}(z), \mathcal{C}(f)) - \log \ka \geq \log \be - \log \ka.$$
	We apply Proposition~\ref{propMane} from $i_s + 1$ to $n$ and obtain
	$$\log |(f^{n-i_s - 1})'(f^{i_s + 1}(z)| > (n - i_s - 1) \ga' - \ga_0' \geq (n - i_s) \si - \ga' - \ga_0'.$$
	
	If $i_s \notin I$, then we apply Proposition~\ref{propMane} from $i_s + k_s + 1$ to $n$. We conclude by summing up all above estimates.
\end{proof}

Recall that $\eta, \iota$ characterize the (FA) condition on $f$. Shrinking $\eta$ only improves condition (FA), so we may assume $\eta \leq \de_2$.

\begin{lemma}
	\label{lemAplus}
	There exists $\rho > 0$ such that for all $v \in f(\mathcal{C}(f))$ and $m > 0$, 
	there exists $G \se \N \cap (m, 2m]$ with
	\begin{equation}
	\label{equDens}
	\frac{|G|}m > 1 - \frac{4 \iota}{\ga }\left(1 + 1.2 \frac{\log \ka}{\si} \right) ,
	\end{equation}	
	such that for all $n \in G$
	\begin{equation}
	\label{equAplusRho}
	|(f^n)'(v)| \cdot a^+(v, n; f) > \rho.
	\end{equation}
\end{lemma}

\begin{proof} 
	Fix $v \in f(\mathcal{C}(f))$ and let $v_j := f^j(v)$ for $j \geq 0$. For all $n > 0$ let $R(n):=\{0 \leq j < n\ :\ v_j \in C(f, \eta)\} = \{ j_1 < j_2 < \ldots < j_s\}$. Note that $s$ depends on $n$.
	Note that by (FA), for all $n > 0$ 
	\begin{equation}
	\label{equInfWR}
	\sum_{j \in R(n)} \ln |f'(v_j)| > -n \iota.
	\end{equation}
	For $i=1,\ldots,s$, let $k_i:= k(v_{j_i})$, $J_i': = \N \cap (j_i, j_i+k_i+1]$ and $J'(n):=\cup_{i=1}^s J_i'$. For any $m \geq 0$ we define
	$$G':=\N \cap (m, 2m] \sm J'(2m).$$
	Let $\th > 1$, $J_i := \N \cap (j_i, j_i + \th(k_i+1)]$ and $J(n):=\cup_{i=1}^s J_i$. As before,
	$$G:=\N \cap (m, 2m] \sm J(2m).$$
	Let us denote
	$$\d_j^n := \log |(f^{n-j})'(v_j)|, \ \d_j := \d_j^{j+1} \text{ and } \d^n := \d_0^n.$$
	By Lemma~\ref{lemKx}(1), $k_i + 1 < -2\ga^{-1}\d_{j_i}$, thus by \requ{InfWR}
	$$|J'(2m)| < 4\frac{m\iota}{\ga} \text{ and }|J(2m)| < 4\frac{m\iota\th}{\ga}.$$
	It suffices to choose $$\th := 1 + 1.2 \frac{\log \ka} {\si}$$ and show inequality \requ{AplusRho} for all $n \in G \se G'$.
	
	We employ Lemma~\ref{lemDfn} to estimate $\d_j^n$ for all $j < n$. First, observe that as $n \in G'$, $\be \geq \eta$ for all $0 \leq j < n$. Therefore, for all $j < n$ we have
	\begin{equation}
	\label{equClean}
	\d_j^n > \log\eta + (n-j)\si -\si_0.
	\end{equation}
	
	By Lemma~\ref{lemKx}(2) and \requ{Clean} for $j=j_i + k_i + 1$
	$$\d_{j_i} + \d_{j_i}^n > 0.1 \d_{j_i} + d_{j_i + k_i + 1}^n > 0.1 \d_{j_i} + (n - j_i - k_i - 1)\si - \si_0.$$
	If $n \in G$, then by Lemma~\ref{lemKx}(1) $n - j_i - k_i - 1 > (\th - 1) (k_i + 1) > -\d_{j_i} \frac{\th - 1}{\log \ka}$.
	Combine the previous two inequalities with the choice of $\th$ to obtain
	\begin{equation}\label{equHit}
	\d_{j_i} + \d_{j_i}^n > - \d_{j_i} + (n - j_i - \th(k_i + 1)) \si > -d_{j_i}, 
	\end{equation}
	as $-d_{j_i} > -\log \eta - \log \ka$.
	
	To complete the proof, using inequality \requ{Apd}, we need to bound
	$$S(n) := \sum_{j=0}^{n-1} \frac{1}{|(f^{n-j})'(v_j)| \cdot |f'(v_j) )|} = \sum_{j=0}^{n-1} \exp(- d_j - d_j^n).$$
	Let $S_0(n) := \sum\limits_{j \in R(n)} \exp(- d_j - d_j^n)$ and $S_1(n) := S(n) - S_0(n)$. By \requ{Clean} and (K2), 
	$$S_1(n) < \ka \eta^{-2} e^{\si_0} \sum_{j=0}^{n-1} e^{-j\si}.$$
	
	Let $Q(n) := \{j \in R(n)\ :\ -d_j > (n - j)\si / 2\}$. Then by inequalities \requ{Hit} and \requ{Clean} we can estimate
\begin{align*}
S_0(n) & = \sum_{j \in Q(n)} \exp(- d_j - d_j^n) + \sum_{j \in R(n) \sm Q(n)} \exp(- d_j - d_j^n)\\
&< \sum_{j \in Q(n)} e^{- \frac{(n - j) \si}2} + \sum_{j \in R(n) \sm Q(n)} e^{-\frac{(n - j) \si} 2 - \log \eta + \si_0} < \eta^{-1} e^{\si_0} \sum_{j=0}^{n-1} e^{- j\frac{\si}2}.
\end{align*}
This ends the proof.
\end{proof}

%
%	$$S_0(n) = \sum_{j \in Q(n)} \exp(- d_j - d_j^n) + \sum_{j \in R(n) \sm Q(n)} \exp(- d_j - d_j^n)$$
%	$$< \sum_{j \in Q(n)} e^{- \frac{(n - j) \si}2} + \sum_{j \in R(n) \sm Q(n)} e^{-\frac{(n - j) \si} 2 - \log \eta + \si_0} < \eta^{-1} e^{\si_0} \sum_{j=0}^{n-1} e^{- j\frac{\si}2}.$$

\begin{remark}\normalfont
The scale $\rho$ only depends on $\si, \si_0$ and $\eta$.
\end{remark}

\subsection{Using the transversality property}\label{section_trans}

From now on, we consider a local parametrization of the moduli space
near $f$, given by a holomorphic injection $F: O \times \P^1 \to \P^1$ where $O$ is a small neighborhood 
of the origin in $\C^{2d-2}$, and $F(0,\cdot)=f$. For all $\lambda \in O$, we denote $f_\lambda:=F(\lambda,\cdot)$.

In  \cite{Astorg}, the first author proves a transversality property
under the assumption that $f$ has no invariant line fields. We show that it is the case for Collet-Eckmann
rational maps which is not a flexible Latt\`es map. Though it is known by the experts, we did not find a complete proof in the literature, so we provide it here.

\begin{lemma}\label{lm:noinvlinfield}
Let $f\in\rat_d$ be Collet-Eckmann. If $f$ carries an invariant line field on its Julia set, then $f$ is a flexible Latt\`es map.
\end{lemma}

\begin{proof}
Recall that a point $z_0\in\J_f$ is said to be a \emph{conical point} of $f$ if there exist $z_j\rightarrow z_0$, integers $n_j\rightarrow+\infty$ and positive numbers $\rho_j\rightarrow 0$ such that $f^{n_j}(z_j+\rho_j\cdot z)$ converges uniformly on $\D$ to some non-constant holomorphic function (see, e.g., \cite{rudiments} page 109). Denote by $\Lambda_f$ the set of conical points of $f$. Przytycki proved (see \cite{Prz} Proposition A.3.4) that for any $f\in\rat_d$ which is a Collet-Eckmann rational map, we have $\Lambda_f=\J_f$. In particular $\mathrm{Leb}
(\Lambda_f)>0$.

As a consequence, if $f$ is Collet-Eckmann and carries an invariant line field on its Julia set, by Theorem VII.22 of \cite{rudiments}, this implies that $f$ is a Latt\`es map. Finally, Corollary 3.18 of \cite{McMullen} asserts that if $f$ is a Latt\`es map and carries an invariant line field on its Julia set, then $f$ is a flexible Latt\`es map.
\end{proof}

Therefore, by Theorem B of \cite{Astorg} and Lemma \ref{lem:calculderivxi}, Collet-Eckmann rational maps satisfy the following property:
\begin{description}
	\item[Transversality (T)]
	The following limits exist (as linear forms on $T_0 O \simeq \C^{2d-2}$), and they are linearly independent:
	$$
	\tau_i :=  u \mapsto \lim_{n \to \infty} \frac{\frac{d}{dt}_{|t=0} \xi_n^i(tu)}{(f^n)'(f(c_i))} = \sum_{j=0}^\infty \frac{\frac{d}{dt}_{|t=0}F(tu, f^j(c_i)) }{(f^j)'(f(c_i))}$$
	where $\xi^i_n(\la):=f_\la^{n+1}(c_i(\la))$. 	
\end{description}

Note that the equality is just the first part of Lemma \ref{lem:calculderivxi} with different notations.
Indeed, observe that $\dot v_i = \frac{d}{dt}_{|t=0} F(tu,c_i)$.

Since the $(\tau_i)$ are linearly independent, we can choose $2d-2$ 
unit vectors $u_1, \ldots, u_{2d-2} \in \C^{2d-2}$ such that $\tau_i(u_j) \neq 0$ if and only if $i=j$. 
Moreover, there is an open subset $E$ of $\mathbb{S}_{2d-2}$ such that for all $u \in E$ and all $i$,
\begin{equation}\label{def:E}
\|\tau_i(u) \| \geq \frac{1}{2d} \min_{j} \|\tau_j(u_j)\|>0
\end{equation}
(we can take $E$ to be a small enough neighborhood of $u={\| \sum_{j=1}^{2d-2} u_j\|}^{-1}\sum_{j=1}^{2d-2} u_j$).

We fix some critical point $c_i$. From the transversality condition (T) and the definition \eqref{def:E}
of the set of directions $E$ in parameter space, we get $m_0>0$ such that for all $m \geq m_0$ and all $u \in E$, $\pa_{u}\xi_m^i(0) \neq 0$. For some  $r \geq 0$, let us define the corresponding parameter neighborhood of $0$ in direction $u$
\begin{equation}\label{equDefV}
V(u,r) := \left( \D(0, r) \cdot u\right) \cap O.
\end{equation}
For any $m \geq m_0$ and each $i$, we say that $\Ga^u( m) \geq 0$ satisfies (T1), (T2) and (T3) if:
\begin{description}
	\item[(T1)] if $V := V(u,\Ga^u( m))$, then for all $m_0 \leq j \leq m$
	$$\Dist(\xi^i_j, V) \leq 10^{-2},$$
	\item[(T2)] $\Si(u,m,V) := \sum\limits_{j=0}^{m-1} \max\limits_{\al, \be \in V} \left|\log \frac {f_\al'(\xi_{j+1}^i(\al))}
	{f_\be'(\xi_{j+1}^i(\be))} \right|\leq 10^{-2}$,
	\item[(T3)] $\diam(\xi_j^i(V)) \leq \frac 1{e \ka}$ for all $0 \leq j \leq m$.
\end{description}
Combining (T1) and (T3) with \eqref{equImg} we observe that 
\begin{equation}
\label{equGaSm}
\Ga^u(m) < |\pa_{u}\xi_m^i(0)|^{-1}.
\end{equation}

Recall that $\tau_i(u) \neq 0$ is a condition of the transversality hypothesis (T)
and the definition \eqref{def:E} of $E$. Denote by $v_i = v_i(0) := f_0(c_i(0))$.
\begin{lemma}
	\label{lemParDist}
	There exist $l_1 > 1$ and $m_1 > m_0$ such that for all $u\in E$ and $m \geq m_1$, then
	\begin{equation}
	\label{equGa}
	\Ga^u( m) = l_1^{-1} \cdot a^+(v_i, m; f_0),
	\end{equation}
	satisfies (T1), (T2) and (T3).
\end{lemma}

\def\c1{11e}
\def\cj{{($*$)$_j$}}

\begin{proof} 
	By (CE) and (T), there exists $m'>m_0$ such that for all  $u\in E$:
	\begin{equation}\label{equ52}
	\ka M\sum_{j=m'}^\infty  e^{-j\gamma + \gamma_0 + 3} < 10^{-3},
	\end{equation}
	\begin{equation}\label{equ54}
	\left| \sum_{j=0}^{m'-1} \frac{\frac{d}{dt}_{|t=0} F(tu,f_0^j(c_i))}{\tau_i(u) (f_0^m)'(v_i)} - 1 \right| < 10^{-3},
	\end{equation}
	where $M := 1 + \max\{|\tau_i(u)|, |\tau_i(u)|^{-1}\ : \ u\in E\}$. Furthermore, for $k \geq m'$, one has by (T) and (CE) that
	\begin{equation}
	\label{equComp}
	\left|\log \frac{\pa_{u}\xi_k^i(0)}{\tau_i(u) (f_0^k)'(v_i)} \right| < 10^{-2} \text{ and } |\pa_{u}\xi_k^i(0)| > 2e\ka.
	\end{equation}	
	We choose $0 < \eps_1 $ small enough such that 
	$$\diam(\xi^i_j(V(u, 2\eps_1)) < e^{-1}\ka^{-1} \text{ for any $u\in E$ and all } 0 \leq j \leq m',$$
	$$\Dist(\xi_j^i, V(u, 2\eps_1)) < 10^{-3} \text{ for any $u\in E$ and all } m_0 \leq j \leq m' \text{ and }$$
	$$\Si(u,m', V(u, 2\eps_1)) < 10^{-3},$$
	recalling that $V$ was defined by \requ{DefV} and $\Si$ in condition (T2).
	Put $$l_1:=\c1\ka M$$ and choose $m_1 > m'$ such that
	$$e^{m_1\gamma_1 - \gamma_0} > \max\left(4e\ka M,\eps_1^{-1}\right).$$
	We prove the lemma for these $l_1$ and  $m_1$.

	Take any $m \geq m_1$, $u\in E$. Let
	$$\La := V(u, l_1^{-1}a^+(v_i, m; f_0)),$$
	$$W := \D(v_i, 10^{-1}a^+(v_i, m; f_0)),$$
	and prove by induction the following claim for all $m_0 \leq j \leq m$
	\begin{enumerate}
		\item[\cj] $\Dist(\xi_j^i, \La) < 10^{-2}$ and $\Si(u, j, \La) < 10^{-2}$.
	\end{enumerate}
	Since $a^+(v_i, m; f) < |(f_0^m)'(v_i)|^{-1} < \eps_1$ by the choice of $m_1$, we can get {\cj} for $m_0 \leq j \leq m'$ from the choice of $\eps_1$.
	
	Assume that for some $m' < j \leq m$, ($*$)$_k$ holds for all $m_0 \leq k < j$. Observe that $\xi_k^i(\La) \se f_0^k(W)$ for all $m' \leq k < j$. On one hand, we use \requ{Comp}, $\Dist(\xi_k^i, \La) < 10^{-2}$, \requ{Img} and the choice of $l_1$. On the other, we use Lemma~\ref{lemDist} applied to $f_0^m$ on $W$ and \requ{Img}. In particular, by Lemma~\ref{lemDist}, for all $\al \in \La$ and $m' \leq k < j$
	\begin{equation} \label{equDfxi}
	\left| \log \frac{f_0'(\xi_k^i(\al))} {f_0'(\xi_k^i(0))} \right| < 10^{-3}.
	\end{equation}
	Summing \requ{Temp1} for $k=m'$ to $j-1$ implies:
	\begin{equation*}
	\sum_{k=m'}^{j-1} \Dist(f_0, f_0^k(W))  < 10^{-3}  a^+(v_i,m;f)  400 e \ka^2\sum_{k=m'}^{j-1}\frac{|(f_0^k)'(v_i)|}{|f_0'(f_0^k(v_i))|},
	\end{equation*}
	hence, by definition of $a^+(v_i,m;f)$:
	\begin{equation}\label{equTemp111}
	\sum_{k=m'}^{j-1} \Dist(f_0, f_0^k(W))  < 10^{-3}.
	\end{equation}
	%which by \requ{PCE} gives
	%$$|(f_\la^k)'(v_i(\la))|^{-1} < |f_\la'(\xi_k^i(\la))| e^{-(k+1)\nu \ga_1 + \ga_0 + 1} < |f_\la'(\xi_k^i(\al))| e^{-(k+1)\nu \ga_1 + \ga_0 + 2}.$$
	Also, for all $\al \in \La$ and $m' \leq k < j$, by (K1), ($*$)$_k$, \requ{Comp} and direct calculation, we get
	$$|f_\al'(\xi_k^i(\al)) - f_0'(\xi_k^i(\al))| < \ka |\al | < \ka l_1^{-1} a^+(v_i, m; f) < \ka |(f_0^k)'(v_i)|^{-1},$$
	$$\left| \frac{\pa_{u}\xi_{k+1}^i(\al)}{\pa_{u}\xi_{k}^i(\al)} - f_\al'(\xi_k^i(\al)) \right|=
	\frac {|\frac{d}{dt}_{|t=0} F(\al+tu,\xi_{k}^i(\al))|} {|\pa_{u}\xi_{k}^i(\al)|} < \frac{\ka e^{0.02}}{|\tau_i(u) (f_0^k)'(v_i)|}.$$
	Divide the last two inequalities by $|f_0'(\xi_k^i(\al))|$, use $\Dist(f_0, \xi_k^i(\La)) < 10^{-3}$, the triangle inequality, (CE) and \requ{52} to obtain
	$$\left| \frac{\pa_{u}\xi_{k+1}^i(\al)}{\pa_{u}\xi_{k}^i(\al)f_0'(\xi_k^i(\al))} - 1 \right| < \frac{\ka M e}{|(f_0^{k+1})'(v_i)|} < \ka M e^{-(k+1)\gamma_1 + \gamma_0 + 2} < 10^{-3},$$
	which immediately implies
	$$\left| \log \frac{\pa_{u}\xi_{k+1}^i(\al)}{\pa_{u}\xi_{k}^i(\al)} - \log f_0'(\xi_k^i(\al)) \right| < \ka M e^{-(k+1)\gamma_1 + \gamma_0 + 3}.$$
	Observe that 
	$$\Dist(\xi_{k+1}^i, \La) \leq \Dist(\xi_{k}^i, \La) + \sup_{\al, \be \in \La} \left| \log \frac{\pa_{u}\xi_{k+1}^i(\al)}{\pa_{u}\xi_{k}^i(\al)} - \log \frac{\pa_{u}\xi_{k+1}^i(\be)}{\pa_{u}\xi_{k}^i(\be)} \right|.$$
	The last two inequalities with  the choice of $\eps_1$, inequality \requ{Temp111}, the choice of $W$ and bound \requ{52} imply:
	\begin{align*}\Dist(\xi_j^i, \La) \leq & \Dist(\xi_{m'}^i, \La) + \sum_{k=m'}^{j-1} \Dist(f_0, f_0^k(W)) \\
	&  \ + 2\sum_{k=m'}^{j-1} \ka M  e^{-(k+1)\gamma_1 + \gamma_0 + 3} < 4\cdot 10^{-3}.
	\end{align*}
	
	Similarly, by the choice of $\eps_1$, inequality \requ{Temp111}, the bound for $|f_\al'(\xi_k^i(\al)) - f_0'(\xi_k^i(\al))|$ above, inequality \requ{Temp1}, the choice of $W$ and bound \requ{52} imply
	$$\Si(u,j, \La) < \Si(u, m', \La) + \sum_{k=m'}^{j-1} \left( \Dist(f_0, f_0^k(W)) + 2\ka e^{-k\gamma_1+\gamma_0+1}\right) < 3\cdot 10^{-3}. $$
	Therefore {\cj} holds for all $m_0 \leq j \leq m$. Thus by inequality \requ{Comp}, the definition of $\La$ and inequality \requ{Apd}
	$$\diam(\xi_j^i(\La)) < e M|(f_0^j)'(v_i)| \diam(\La) < e^{-1}\ka^{-1}.$$
	Conditions (T1-3) are satisfied by $\Ga^u(m) := l_1^{-1} \cdot a^+(v_i, m; f)$, which completes the proof.
\end{proof}

\subsection{Proof of Proposition \ref{th:distortionestimate}}

\begin{proof}[Proof of Proposition \ref{th:distortionestimate}] 
	We assume that we have a rational map $f=f_0$ that satisfies all the assumptions of Proposition~\ref{th:distortionestimate} and such that
	$$\theta := 1 - \frac{4 \iota}{\ga }\left(1 + 1.2 \frac{\log \ka}{\si} \right) > 1 - \frac{1}{2d - 2},$$
	where $\si$ is defined by \requ{Si}.
	
	We apply Lemma~\ref{lemParDist}. 
	Fix $\rho > 0$ provided by Lemma~\ref{lemAplus}, applied for each $i=1,\ldots,2d-2$. The intersection $G$ of sets $G_1,\ldots,G_{2d-2}$ of densities at least $\theta$ in $\N \cap (m,2m]$ has positive density, thus the union of such sets $G$ is infinite.
	
	By inequalities \requ{Ga}, \requ{AplusRho} and \requ{Comp}, for each $n \in G$ and $i=1,\ldots,2d-2$ we compute, for $u \in E$:
	$$\Ga^i( n) = l_1^{-1} \cdot a^+(v_i, n; f) > \frac{l_1^{-1} \rho}{|(f^n)'(v_i)|} > \frac{l_1^{-1} \rho M}{e |\pa_{u}\xi_n^i(0)|},$$
	where $M := 1 + \max\{|\tau_i(u)|, |\tau_i(u)|^{-1}\ : \ u\in E\}$. 
	%	Using the distortion bound (T1) and inclusion \requ{Img}, we obtain that 
	%	$$\D\left(\xi_n^i(0), e^{-2} l_1^{-1} \rho M^{-1} \right)\ \se\ \xi_n^i(\D(0, \Ga^i( n)) \cdot u).$$
	The bounds for $r_{n,i}:=\Ga^i( n)$ compared to $C:=e^{2} l_1  M \rho^{-1} $ follow from the bounds \requ{Comp}, \requ{Img} and (T1) in the definition of $\Ga^i( n)$.
	This gives the lower bound for $r_{n,i} |(f^n)'(v_i)|$.
On the other hand, \eqref{equGaSm} gives
\[r_{n,i} |(f^n)'(v_i)|<\frac{|(f^n)'(v_i)|}{|\pa_{u}\xi_n^i(0)|}\leq 2 M\]
for all $n\in G$ large enough and
the upper bound follows.
	This completes the proof.
\end{proof}

\begin{theorem}\label{theo_largescale}
	Let $f$ be a rational map satisfying the conditions of Proposition \ref{th:distortionestimate}. Then $f$ satisfies the 
	large scale condition.
\end{theorem}

\begin{proof}	
	Let $f=f_0$ be such a rational map, and let 
	$(f_\lambda)_{\lambda \in O}$ be a local holomorphic parametrization of the moduli space,
	where $O$ is some neighborhood of the origin in $\C^{2d-2}$. 
	Recall that by the transversality property (T), there are unit vectors $(u_j)_{1 \leq j \leq 2d-2}$ 
	in $\C^{2d-2}$ such that $\tau_i(u_j)\neq0$ if and only if $i=j$. Let $r_{n,i}$, $G\subset \N$ and $E\subset\mathbb{S}_{2d-2}$ be provided by Proposition \ref{th:distortionestimate}. Recall that $E$ is an open neighborhood
	of $\| \sum_{j=1}^{2d-2}u_j\|^{-1}\sum_{j=1}^{2d-2} u_i$ in the unit sphere $\mathbb{S}_{2d-2}$ of $\C^{2d-2}$, such that for all $1 \leq j \leq 2d-2$
	we have $\tau_j(u) \neq 0$.
	We now fix an affine chart on $\mathbb{P}^1$ for the rest of the proof.
	
	For all $n \in G$, let
	\begin{equation*}
	\phi_n(\lambda):=\left(\xi_{n}^i(r_{n,i}\lambda) - \xi_n^i(0)\right)_{1 \leq i \leq 2d-2}
	\end{equation*}
	As $r_{n,i}\asymp |(f^n)'(v_i)| ^{-1}\to 0$ when $n \to \infty$, for all $i$ by transversality property (T), for all $n\in G$ large enough
	$\phi_n$ is defined in the neighborhood of a ball of fixed size in $\C^{2d-2}$. Up to rescalling, we may assume it is the unit ball 
	$\B \subset \C^{2d-2}$.
	Observe that for all $1 \leq i \leq 2d-2$, we have
	\begin{equation*}
	\frac{d}{dt}_{|t=0} \phi_n^i(tu) = r_{n,i} D\xi_n^i(0) \cdot u
	\end{equation*}
	so that by the definitions of $\tau_i$ we have
	\begin{equation*}
	\frac{d}{dt}_{|t=0} \phi_n^i(tu) = r_{n,i} \tau_i(u) (f^n)'(v_i) +o(1), \text{ as  }n\to\infty.
	\end{equation*}
	Using again Proposition~\ref{th:distortionestimate} and inequality \eqref{def:E}, we find that there is a constant 
	$C'>0$ depending only on $f$ such that for all $n \in G$ and all $u\in E$,
	\begin{equation*}
	\frac{1}{C'} \leq \left|	\frac{d}{dt}_{|t=0} \phi_n^i(tu) \right| \leq C'.
	\end{equation*}

	From the distortion control of Proposition \ref{th:distortionestimate} and equation \eqref{equImg}, 
	we therefore have that for all $u \in E$, for all $t \in \D$, and for all $1\leq i\leq2d-2$,
	\begin{equation*}
	|\phi_n^i(tu)| \leq \frac{11}{10} C'.
	\end{equation*}

	Now applying Sibony-Wong's theorem (Theorem \ref{tm:sibonywong}) to $\phi_n$ on $E$, we see that there is a constant $0<\tilde \rho<1$ 
	independent 
	of $n \in G$ such that for all $z \in  \B(0,\tilde \rho)$, 
	\begin{equation*}
	\| \phi_n(z)\| \leq (2d-2)\frac{11}{10} C'.
	\end{equation*}
	Therefore the sequence $(\phi_n)_{n \in G}$ is normal near the origin, hence up to extraction
	converges to some holomorphic map $\phi: \B(0, \tilde \rho) \to \C^{2d-2}$.
	By the transversality property (T) and the fact that 
	$r_{n,i} |(f^n)'(v_i)| \geq \frac{1}{C}$, we have that $D\phi(0)$ is invertible. 
	Therefore $\phi(\B(0, \tilde \rho))$ contains 
	some ball $\B(0,2r)$ for some $r>0$, and so for all large enough $n \in G$, $\phi_{n}(\B(0, \tilde \rho))$ 
	contains the ball $\B(0,r)$.
	Recalling the definition of $\phi_{n}$, we claim that this implies that $f$ satisfies the large scale condition for all $n\in G$
	with 
	$$\Omega_{n}:=\B\left(0,\max_{1\leq i \leq 2d-2} r_{n,i}\right).$$
Indeed, since $D\phi(0)$ is invertible, in particular, the graph of $\phi$ is vertical-like in a neighborhood of the origin, restricting $r$ if necessary; by normality, this is also the case for $\phi_n$ for $n$ large enough which is exactly what we want.  
\end{proof}

\section{The proofs of Theorem~\ref{tm:leb} and Corollary~\ref{tm:approxhyperbolic}}\label{sec:mainresults}

%This section is devoted to proving Theorem~\ref{tm:leb} and Corollary~\ref{tm:approxhyperbolic}. We begin with proving Theorem~\ref{tm:leb}. We then give a more precise approximation result for good Collet-Eckmann rational maps as Theorem~\ref{tm:approx}. This result is just the combination of Theorem~\ref{tm:suppTm} and Theorem~\ref{tm:bigscale} with classical results concerning the bifurcation currents. Corollary~\ref{tm:approxhyperbolic} is just an immediate consequence of both results.

A map $f\in\rat_d$ is a \emph{flexible Latt\`es map} if there exist an elliptic curve $E$, an affine map $\ell:E\rightarrow E$ with integral linear part and finite branched cover $\pi:E\rightarrow\P^1$ such that the following diagram commutes
\[ 	\begin{tikzcd}
E \arrow{r}{\ell} \arrow{d}{\pi} & E  \arrow{d}{\pi}\\
\P^1 \arrow{r}{f} &\P^1
\end{tikzcd}\]
It follows from the definition that any flexible Latt\`es map is strongly Misiurewicz, that $\mathcal{P}(f)\cap\mathcal{C}(f)=\varnothing$ and that $\mathrm{Card}(\mathcal{P}(f))\leq 4$ and that its degree $d$ is the square of the linear part of $\ell$ (see e.g.~\cite{Milnor}). In particular, if $f$ is a flexible Latt\`es map, then $d\geq4$.

\begin{proof}[Proof of Theorem~\ref{tm:leb}]
Recall that a rational map is strongly Misiurewicz if all its critical points are preperiodic to repelling cycles. 
Let us denote by $\mathcal{X}\subset\mathcal{M}_d$ the set of all conjugacy classes of strongly Misiurewicz degree $d$ rational maps and by $\mathcal{X}^*\subset\mathcal{X}$ those classes of rational maps with $\mathcal{P}(f)\cap \mathcal{C}(f)=\varnothing$ with simple critical values. Choose $[f]\in\mathcal{X}^*$. If $d\leq 3$, then $f$ can not be a flexible Latt\`es map.
If $d=\deg(f)\geq4$, since $f$ has simple critical values, $\mathrm{Card}(\mathcal{P}(f))\geq \mathrm{Card}(f(\mathcal{C}(f)))=2d-2>4$ hence $\mathcal{X}^*$ does not contain flexible Latt\`es maps. Observe that $\mathcal{X}^*$ is countable.

~

Pick any $[f]\in\supp(\mu_\bif)$ and $\Omega\subset\mathcal{M}_d$ an open neighborhood of $[f]$. According to \cite[Main Theorem]{buffepstein}, the set $\mathcal{X}^*$ is dense in $\supp(\mu_\bif)$, hence in particular $\mathcal{X}^*\cap\Omega\neq \varnothing$. 
Choose now such $[f_0]\in \mathcal{X}^*\cap \Omega$.  By the above argument, we apply Theorem~\ref{tm:abundance}:
there exist $\mu(f_0),\mu_0(f_0),\gamma(f_0),\gamma_0(f_0)>0$ and $\hat{\alpha}(f_0)>0$ such that for all $\alpha < \min(\frac{\gamma(f_0)}{200}, \hat \alpha(f_0))$, 
there exist $\hat{\eta}(f_0)>0$ and $\hat{\iota}(f_0)>0$ such that for all $\eta < \hat \eta(f_0)$ and for all
$\iota < \hat \iota(f_0)$, the map $f_0$ is a Lebesgue point of density of rational maps satisfying
CE($\gamma(f_0),\gamma_0(f_0)$), CE2($\mu(f_0),\mu_0(f_0)$), BA($\alpha$) and  FA($\eta, \iota$).

We let $CE_{[f_0]}$ be those conjugacy classes $[g]\in\Omega$ of maps $g$ satisfying the properties CE($\gamma(f_0),\gamma_0(f_0)$), CE2($\mu(f_0),\mu_0(f_0)$), BA($\alpha$) and  FA($\eta, \iota$) with $\eta < \hat \eta(f_0)$ and $\iota < \hat \iota(f_0)$. 
Up to reducing $\hat{\iota}(f_0)$, we may assume that such $[g]$ satisfies the Large scale condition by Theorem~\ref{tm:bigscale}.
The set $CE_{[f_0]}$ is contained in $\Omega\cap\supp(\mu_\bif)$ by Theorem~\ref{tm:suppTm} and has positive Lebesgue volume by Theorem~\ref{tm:abundance}. This concludes the proof of Theorem~\ref{tm:leb}.
\end{proof}

Combining Theorem~\ref{tm:bigscale} with Theorem~\ref{tm:suppTm}, we also get the following result. In the light of the proof above, it is an enhanced version of Corollary \ref{tm:approxhyperbolic}.

\begin{theorem}\label{tm:approx}
Let $\gamma, \gamma_0, \mu, \mu_0>0$, let $\alpha<\gamma/200$ and let $\eta>0$.
Let $\iota>0$ be given by Theorem~\ref{tm:bigscale}. Assume that $f\in\rat_d$ has simple critical points and satisfies CE($\gamma,\gamma_0$), CE2($\mu,\mu_0$), BA($\alpha$) and  FA($\eta, \iota$), then $[f]\in\supp(\mu_\bif)$. In particular,
\begin{enumerate}
\item $f$ is approximated by strongly Misiurewicz rational maps;
\item for any given $\theta_1,\ldots,\theta_{2d-2}\in\R\setminus2\pi\Z$, $f$ is approximated by rational maps having $2d-2$ distinct neutral cycles with respective multipliers $e^{i\theta_1},\ldots,e^{i\theta_{2d-2}}$;
\item for any given $w_1,\ldots,w_{2d-2}\in\D$, $f$ is approximated by hyperbolic rational maps with $2d-2$ distinct attracting cycles with respective multipliers $w_1,\ldots,w_{2d-2}$.
\end{enumerate}
\end{theorem}

\begin{proof}
Combining Theorem~\ref{tm:bigscale} with Theorem~\ref{tm:suppTm}, we see that for any such Collet-Eckmann rational map $f$, the class $[f]$ belongs to $\supp(\mu_\bif)$, or equivalently that $f\in\supp(T_\bif^{2d-2})$. The three points respectively follow from \cite[Theorem 1]{buffepstein}, from \cite[Theorem 1.1]{neutral} and from Theorem 134 of \cite{bsurvey}.
\end{proof}

The support of the measure $\mu_\bif$ has been proven to be the closure of classes of rational maps having a maximal number of neutral cycles with given multipliers. The following is a direct consequence of Theorem~\ref{tm:leb} and \cite[Theorem 1.1]{neutral}.

\begin{corollary}\label{positivity_neutral}
Fix any $\theta_1,\ldots,\theta_{2d-2}\in\R\setminus2\pi\Z$, and let $\mathcal{N}(\theta_1,\ldots,\theta_{2d-2})$ be the set of degree $d$ rational maps having $(2d-2)$ distinct neutral cycles of respective multipliers $\exp(i\theta_1),\ldots,\exp(i\theta_{2d-2})$. Then
\[\mathrm{Vol}_{\mathcal{M}_d}\left(\overline{\{[f]\in\mathcal{M}_d\, ; \ f\in \mathcal{N}(\theta_1,\ldots,\theta_{2d-2})\}}\right)>0.\]
\end{corollary}

It is actually difficult to exhibit one example of such a rational map. Here, we prove that they are dense in a set which is \emph{not} negligible.

 \bibliographystyle{short}
 \bibliography{biblio}
\end{document}